\documentclass[11pt, twoside, leqno]{article}

\usepackage{amssymb}
\usepackage{amsmath}
\usepackage{amsthm}
\usepackage{color}
\usepackage{mathrsfs}

\allowdisplaybreaks

\def\rr{{\mathbb R}}

\def\rd{{\mathbb{R}^d}}
\def\nn{{\mathbb N}}
\def\zz{{\mathbb Z}}
\def\cc{{\mathbb C}}

\def\cx{{\mathcal X}}
\def\fz{\infty}
\def\az{\alpha}
\def\bz{\beta}
\def\dz{\delta}

\def\ez{\epsilon}
\def\gz{{\gamma}}

\def\lz{\lambda}

\def\oz{\omega}

\def\vz{\varphi}

\def\lf{\left}
\def\r{\right}

\def\hs{\hspace{0.25cm}}
\def\ls{\lesssim}

\def\noz{\nonumber}
\def\wz{\widetilde}

\def\st{\subset}

\def\loc{{\mathop\mathrm{\,loc\,}}}
\def\supp{\mathop\mathrm{\,supp\,}}

\def\dint{\displaystyle\int}
\def\rbmo{\mathop\mathrm{\,{\rm \widetilde{RBMO}}(\mu)}}
\def\kbsp{\widetilde{K}_{B,\,S}^{(\rho)}}
\def\kbrp{\widetilde{K}_{B,\,R}^{(\rho)}}
\def\dfrac{\displaystyle\frac}

\newtheorem{theorem}{Theorem}[section]
\newtheorem{lemma}[theorem]{Lemma}
\newtheorem{corollary}[theorem]{Corollary}
\newtheorem{proposition}[theorem]{Proposition}

\theoremstyle{definition}
\newtheorem{remark}[theorem]{Remark}
\newtheorem{definition}[theorem]{Definition}
\renewcommand{\appendix}{\par
   \setcounter{section}{0}%
   \setcounter{subsection}{0}%
   \setcounter{subsubsection}{0}%
   \gdef\thesection{\@Alph\c@section}%
   \gdef\thesubsection{\@Alph\c@section.\@arabic\c@subsection}%
   \gdef\theHsection{\@Alph\c@section.}%
   \gdef\theHsubsection{\@Alph\c@section.\@arabic\c@subsection}%
   \csname appendixmore\endcsname
 }

\numberwithin{equation}{section}

\begin{document}

\arraycolsep=1pt

\title{\bf\Large Boundedness of Commutators on Hardy Spaces over
Metric Measure Spaces of Non-homogeneous Type
\footnotetext{\hspace{-0.35cm} 2010
{\it Mathematics Subject Classification}. {Primary 47B47; Secondary
42B20, 42B35, 30L99.}
\endgraf{\it Key words and phrases.} geometrically doubling measure,
upper doubling measure, metric measure space of non-homogeneous type,
commutator, Calder\'on-Zygmund operator, fractional integral,
$\rm{\widetilde{RBMO}}$ function, Hardy space.
\endgraf
Haibo Lin is supported by the National Natural Science Foundation of
China (Grant Nos. 11301534 and 11471042) and Da Bei Nong Education
Fund (Grant No. 1101-2413002). Dachun Yang is supported by the
National Natural Science Foundation of China
(Grant Nos. 11571039 and 11361020),
the Specialized Research Fund for the Doctoral Program of
Higher Education of China (Grant No. 20120003110003)
and the Fundamental Research Funds for Central Universities of China
(Grant Nos. 2013YB60 and 2014KJJCA10).
}}
\author{Haibo Lin, Suqing Wu and Dachun Yang\,\footnote{Corresponding author}}
\date{}
\maketitle

\vspace{-0.8cm}

\begin{center}
\begin{minipage}{13cm}
{\small {\bf Abstract.}\quad Let $(\mathcal{X},d,\mu)$ be a
metric measure space satisfying the so-called upper doubling
condition and the geometrically doubling condition. Let $T$ be a
Calder\'{o}n-Zygmund operator with kernel satisfying only
the size condition and some H\"ormander-type condition,
and $b\in\rm{\widetilde{RBMO}(\mu)}$ (the regularized BMO
space with the discrete coefficient). In this paper, the authors
establish the boundedness of the commutator $T_b:=bT-Tb$ generated
by $T$ and $b$ from the atomic Hardy space $\widetilde{H}^1(\mu)$
with the discrete coefficient into the weak Lebesgue space
$L^{1,\,\infty}(\mu)$. The boundedness of the commutator
generated by the generalized fractional integral
$T_\alpha\,(\alpha\in(0,1))$ and the $\rm{\widetilde{RBMO}(\mu)}$
function from $\widetilde{H}^1(\mu)$ into
$L^{1/{(1-\alpha)},\,\fz}(\mu)$ is also presented. Moreover,
by an interpolation theorem for sublinear operators, the authors
show that the commutator $T_b$ is bounded on $L^p(\mu)$ for all
$p\in(1,\infty)$.
}
\end{minipage}
\end{center}

\section{Introduction}\label{s1}
\hskip\parindent
The classical theory of Calder\'{o}n-Zygmund operators
originated from the study of the convolution operator
with singular kernel on $\mathbb{R}$. From then on,
it has become one of the core research areas
in harmonic analysis and has been developed
into a large branch of analysis on metric spaces,
among which, one of the most useful underlying spaces
is the space of homogeneous type introduced by Coifman
and Weiss \cite{cw71, cw77}. Recall that a quasi-metric space
$(\cx, d)$ equipped with a non-negative measure
$\mu$ is called  a {\it space of homogeneous type} in
the sense of Coifman and Weiss \cite{cw71, cw77} if
$(\cx,d,\mu)$ satisfies the
{\it measure doubling condition}:
there exists a positive constant $C_{(\mu)}$ such that,
for all balls
$B(x,r):= \{y\in\cx:\,d(x, y)< r\}$ with
$x\in\cx$ and $r\in(0, \fz)$,
\begin{equation}\label{1.1}
\mu(B(x, 2r))\le C_{(\mu)} \mu(B(x,r)).
\end{equation}
As was well known, the space of homogeneous type is a natural
setting for Calder\'{o}n-Zygmund operators and function spaces.
Euclidean spaces equipped with Lebesgue measures, Euclidean
spaces equipped with weighted Radon measures satisfying
the doubling condition \eqref{1.1}, Heisenberg groups
equipped with left-variant Haar measures are all
the typical examples of spaces of homogeneous type.

On the other hand, in the last two decades, many classical
results concerning the Calder\'{o}n-Zygmund operators and
function spaces have been proved still valid for metric
spaces equipped with \emph{non-doubling measures};
see, for example,
\cite{ntv98, ntv03, t01, t012, t03, t0302, t14, cmy05,
cs02, hmy05,hyy09}. In particular, let $\mu$ be a
non-negative Radon measure on $\mathbb{R}^d$
which only satisfies the \emph{polynomial growth condition}
that there exist some positive constant $C_0$ and
$n\in(0,d\,]$ such that, for all $x\in\mathbb{R}^d$
and $r\in(0,\fz)$,
\begin{equation}\label{1.2}
\mu(B(x,r))\leq C_0r^n,
\end{equation}
where $B(x,r):=\{y\in{\mathbb{R}^d}:\,|x-y|<r  \}$.
Such a measure does not need to satisfy the doubling
condition \eqref{1.1}. Tolsa \cite{t01,t0302} introduced
the atomic Hardy space $H_{\rm{atb}}^{1,\,q}(\mu)$, for
$q\in(1,\fz]$, and its dual space, $\mathop\mathrm{RBMO}(\mu)$,
the \emph{space of functions with regularized bounded mean
oscillation}, with respect to $\mu$ as in \eqref{1.2},
and proved that Calder\'{o}n-Zygmund operators are bounded
from $H_{\rm{atb}}^{1,\,q}(\mu)$ into $L^1(\mu)$. In \cite{hmy05},
Hu et al. established an equivalent characterization
of $H_{\rm{atb}}^{1,\,q}(\mu)$ to obtain the boundedness on $L^p(\mu)$
of commutators and their endpoint estimates. More research on
function spaces, mainly on Morrey spaces, and their applications
related to non-doubling measures can be found in, for example,
\cite{gs13, ss14, st05, st07-2}. We point out that the analysis
on such non-doubling context plays a striking role in solving
several long-standing problems related to the analytic capacity,
like Vitushkin's conjecture or Painlev\'{e}'s problem;
see \cite{t03, t14}.

However, as was pointed out by Hyt\"onen in \cite{h10},
the measure $\mu$ satisfying the polynomial growth condition is
different from, not general than, the doubling measure.
Hyt\"onen \cite{h10} introduced a new class of metric measure
spaces satisfying both the so-called upper doubling condition
and the geometrically doubling condition (see, respectively,
Definitions \ref{d2.1} and \ref{d2.4} below), which
are also simply called \emph{metric measure spaces of
non-homogeneous type}. These metric measure spaces of
non-homogeneous type include both metric measure spaces
of homogeneous type and metric measure spaces equipped
with non-doubling measures as special cases. We mention
that several equivalent characterizations for the upper
doubling condition were recently established by Tan and
Li \cite{tl, tl-mn}.

From now on, we always assume that $(\cx,\,d,\,\mu)$ is a
metric measure space of non-homogeneous type in the sense of
Hyt\"onen \cite{h10}. In this new setting, Hyt\"{o}nen \cite{h10}
introduced the space $\mathop\mathrm{RBMO}(\mu)$ and established
the corresponding John-Nirenberg inequality, and Hyt\"onen and
Martikainen \cite{hm12} further established a version of
$Tb$ theorem. Later, Hyt\"{o}nen et al.
\cite{hyy12} and Bui and Duong \cite{ad13}, independently,
introduced the atomic Hardy space $H_{\rm{atb}}^{1,\,q}(\mu)$
and proved that the dual space of $H_{\rm{atb}}^{1,\,q}(\mu)$ is
$\mathop\mathrm{RBMO}(\mu)$.
Recently, Fu et al. \cite{fyy12, fyy14t} established the
boundedness of multilinear commutators of Calder\'on-Zygmund
operators and commutators of generalized fractional integrals
with $\mathop\mathrm{RBMO}(\mu)$.
The boundedness of commutators of multilinear singular integrals
on Lebesgue spaces was obtained by Xie et al. \cite{xgz}.
In addition, Fu et al. \cite{fyy14j} introduced a version of
the atomic Hardy space
$\wz H_{\rm{atb},\,\rho}^{1,\,q,\,\gamma}(\mu)$
($\subset H_{\rm{atb}}^{1,\,q}(\mu)$
and simply denoted by $\widetilde{H}^1(\mu)$;
see Definition \ref{d2.11} below)
and its corresponding dual space
$\rbmo$ ($\supset\mathop\mathrm{RBMO}(\mu)$;
see Definition \ref{d2.13} below) via the discrete coefficients
$\kbsp$. Moreover, Hyt\"onen and Martikainen \cite{hm14} proved
a non-homogeneous $T1$ theorem
for certain bi-parameter singular integral operators.
Very recently, Fu et al. \cite{flyy15} partially established the
theory of the Hardy space $H^p$ with $p\in(0,\,1]$ on $(\cx, d, \mu)$.
Sawano et al. \cite{sst} presented an example showing that,
if $(\cx, d, \mu)$ is not geometrically doubling, then Morrey
spaces depend on the auxiliary parameters.
More research on function spaces and the boundedness of
various operators on metric measure spaces of non-homogeneous type
can be found in \cite{ly11, hlyy, ly12, lyy12, hmy13, b13, cj,
ly14, lmy14, ccj}.
We refer the reader to the survey \cite{yyf} and the monograph
\cite{yyh} for more developments on harmonic analysis in this
setting.

The main purpose of this paper is to generalize the corresponding
results in \cite{hmy05} to the present setting $(\cx,d,\mu)$.
Precisely, let $T$ be a Calder\'{o}n-Zygmund operator with kernel
satisfying only the size condition and some H\"ormander-type condition,
and $b\in\rm{\widetilde{RBMO}(\mu)}$. Under the assumption that $T$ is
bounded on $L^2(\mu)$, we obtain the boundedness of the commutator
$$T_b:=bT-Tb,$$
generated by $T$ and $b$, from the atomic Hardy space
$\widetilde{H}^1(\mu)$ into the weak Lebesgue space
$L^{1,\,\infty}(\mu)$. The boundedness of the commutator generated
by the generalized fractional integral $T_\alpha\,(\alpha\in(0,1))$
and the $\rm{\widetilde{RBMO}(\mu)}$ function from
$\widetilde{H}^1(\mu)$ into the weak Lebesgue space
$L^{1/{(1-\alpha)},\,\fz}(\mu)$ is also established.
Moreover, by an interpolation theorem for sublinear operators,
we also show that the commutator $T_b$ is bounded on $L^p(\mu)$
for all $p\in(1,\infty)$.

This paper is organized as follows.  In Section \ref{s2},
we first recall some necessary notation and notions, including
the discrete coefficient $\kbsp$ and its fundamental properties,
the atomic Hardy space $\wz H_{\rm{atb},\,\rho}^{1,\,q,\,\gamma}(\mu)$
(simply denoted by $\widetilde{H}^1(\mu)$) and the space $\rbmo$
with $\kbsp$, and the Calder\'{o}n-Zygmund decomposition.
We also establish an equivalent characterization and the
John-Nirenberg inequality of $\rbmo$
(see, respectively, Lemma \ref{l2.15} and Proposition
\ref{p2.16} below), whose proofs are similar to those of the
corresponding known results of
$\mathop\mathrm{RBMO}(\mu)$, the details being omitted.
Moreover, in this section, we find a useful property
of the dominating function (see Lemma \ref{l2.3} below),
which is of independent interest
and is used in Section \ref{s3}.

In Section \ref{s3}, we establish the boundedness of the
commutator $T_b$ from $\wz H^1(\mu)$ into $L^{1,\fz}(\mu)$ by
borrowing some ideas from \cite[Theorem 4.1]{hmy05} and
applying Lemma \ref{l2.3}.

In Section \ref{s4}, we prove that the commutator, generated
by the generalized fractional integral
$T_\alpha\,(\alpha\in(0,1))$
and the $\rm{\widetilde{RBMO}(\mu)}$ function, is bounded
from $\widetilde{H}^1(\mu)$ into $L^{1/{(1-\alpha)},\,\fz}(\mu)$.
Recall that the fractional type of the discrete coefficient
$\kbsp$ is a useful tool in the study of commutators of fractional
integrals in the setting of metric measure spaces with non-doubling
measures or metric measure spaces of non-homogeneous type;
see, for example, \cite{cs02,fyy14t}. However, in our proof,
via the Minkowski integral inequality and the Fatou lemma,
we do not need to use the fractional coefficient, which is
a different approach  to deal with commutators of  fractional
integrals.

Section \ref{s5} is devoted to the boundedness on $L^p(\mu)$,
with $p\in(1,\fz)$, of the commutator $T_b$.
To this end, we first establish an interpolation theorem for
sublinear operators (see Theorem \ref{t5.5} below).
Although the interpolation theorem is similar to
\cite[Theorem 3.1]{hmy05}, its proof is different.
Precisely, since it is not clear whether or not the operator
$M_r^\sharp\circ T_1$ compounded by the sharp maximal operator
$M_r^\sharp$ and the sublinear operator $T_1$ is quasi-linear,
the method used in the proof of \cite[Theorem 3.1]{hmy05} might
be problematic. To avoid this, in the below proof of
Theorem \ref{t5.5}, we borrow some ideas from the proof of
\cite[Theorem 1.6]{lmy14}. Then we establish a pointwise
estimate for $M_r^\sharp\circ T_b$, which, together with
the interpolation theorem, yields the desired conclusion.

Finally, we make some conventions on notation. Throughout
this paper, we always denote by $C$, $\wz C$, $c$ or
$\wz c$ a \emph{positive constant} which is independent
of the main parameters, but they may vary from line to line.
\emph{Constants with subscripts}, such as $C_0$ and $c_0$,
do not change in different occurrences. Furthermore, we use
$C_{(\az)}$ to denote a positive constant depending on the
parameter $\az$. The expression $Y\ls Z$ means that there
exists a positive constant $C$ such that $Y\leq CZ$.
The expression $A\sim B$ means that $A\ls B\ls A$.
Let $\nn:=\{1,2,\ldots\}$ and $\zz_+:=\{0\}\cup\nn$.
For any ball $B\st\cx$, we denote its \emph{center} and
\emph{radius}, respectively,
by $c_B$ and $r_B$ and, moreover, for any $\rho\in(0,\fz)$,
we denote the ball $B(c_B,\rho r_B)$ by $\rho B$.
Given any $q\in(0,\fz)$, let $q\prime:=q/(q-1)$ denote its
\emph{conjugate index}. Also, for any subset $E\st\cx$,
$\chi_E$ denotes its \emph{characteristic function}.
For any $f\in L_{\loc}^1(\mu)$ and any measureable set
$E$ of $\cx$, $m_E(f)$ denotes its mean over $E$, namely,
$$m_E(f):=\frac1 {\mu(E)}\int_E f(x)\,d\mu(x).$$
For arbitrary $a\in\mathbb{R}$,  $\lfloor a \rfloor$ denotes
the largest integer smaller than or equal to $a$.

\section{Preliminaries}\label{s2}
\hskip\parindent
In this section, we recall some necessary notions and notation,
including the dominating function,
the discrete coefficient $\kbsp$,
the atomic Hardy space
$\wz H_{\rm{atb},\,\rho}^{1,\,q,\,\gamma}(\mu)$,
the space $\rbmo$ and the Calder\'{o}n-Zygmund decomposition.
We also give out a useful property of the dominating function.

The following notion of upper doubling metric
measure spaces was originally introduced
by Hyt\"onen \cite{h10} (see also \cite{hlyy,lyy12}).

\begin{definition}\label{d2.1}
A metric measure space $(\cx,d,\mu)$ is said to be
\emph{upper doubling} if $\mu$ is a Borel measure on $\cx$ and
there exist a \emph{dominating function}
$\lz:\cx \times (0,\fz)\to (0,\fz)$ and a positive constant
$C_{(\lz)}$, depending on $\lz$, such that, for each
$x\in \cx$, $r\to \lz(x,r)$ is
non-decreasing and, for all $x\in \cx$ and $r\in (0,\fz)$,
\begin{equation}\label{2.1}
\mu(B(x,r))\le\lz(x,r)\le C_{(\lz)}\lz(x,r/2).
\end{equation}
\end{definition}

\begin{remark}\label{r2.2}
\begin{itemize}
\item[(i)] Obviously, a space of homogeneous type is a
special case of upper doubling spaces, where we take
the dominating function $\lz(x,r):=\mu(B(x,r))$ for
all $x\in\cx$ and $r\in(0,\fz)$.
On the other hand, the $d$-dimensional Euclidean space
$\rr^d$ with any Radon measure $\mu$ as
in \eqref{1.2} is also an upper doubling
space by taking $\lz(x,r):=C_0r^{n}$ for all $x\in\rr^d$ and
$r\in(0,\fz)$.

\item[(ii)] Let $(\cx,d,\mu)$ be upper doubling with $\lz$ being
the dominating function on $\cx \times (0,\fz)$ as
in Definition \ref{d2.1}. It was proved
in \cite{hyy12} that there exists another
dominating function $\wz{\lz}$ such that
$\wz{\lz}\le \lz$, $C_{(\wz{\lz})}\le C_{(\lz)}$
and, for all $x,\,y\in \cx$ with $d(x,y)\le r$,
\begin{equation}\label{2.2}
\wz{\lz}(x,r)\le C_{(\wz{\lz})}\wz{\lz}(y,r).
\end{equation}

\item[(iii)] It was shown in \cite{tl} that the upper doubling condition
is equivalent
to the \emph{weak growth condition}:
there exist a dominating function $\lz:\,\cx\times(0,\fz)\to(0,\fz)$,
with $r\to\lz(x,r)$ non-decreasing, positive constants $C_{(\lz)}$,
depending on $\lz$, and $\ez$ such that

\begin{itemize}
\item[$\rm (iii)_1$] for all $r\in(0,\fz)$, $t\in[0,r]$, $x,\,y\in\cx$
and $d(x,y)\in[0,r]$,
$$|\lz(y,r+t)-\lz(x,r)|\le C_{(\lz)}\lf[\frac{d(x,y)+t}r\r]^{\ez}
\lz(x,r);$$

\item[$\rm (iii)_2$] for all $x\in\cx$ and $r\in(0,\fz)$,
$\mu(B(x,r))\le\lz(x,r)$.
\end{itemize}
\end{itemize}
\end{remark}

The following property of the dominating function $\lz$ is useful
and of independent interest.

\begin{lemma}\label{l2.3}
Let $(\cx,d,\mu)$ be an upper doubling space with dominating
function $\lz$ satisfying \eqref{2.2}
and ball $B\subset\cx$. Then, for any $x_1,\ x_2\in B$ and
$y\in {\cx\backslash (kB)}$ with $k\in[2,\fz)$,
it holds true that $\lz(x_1,d(x_1,y))\sim\lz(x_2,d(x_2,y))$.
\end{lemma}

\begin{proof}
Without loss of generality, we may assume that
$d(x_1,y)\leq d(x_2,y)$.
By \eqref{2.2} and the fact that
$\lz(x,r)$ is non-decreasing according to $r$, we have
$$\lz(x_1,d(x_1,y))\sim\lz(y,d(x_1,y))
\leq \lz(y,d(x_2,y))\sim\lz(x_2,d(x_2,y)).$$
Therefore, to prove Lemma \ref{l2.3}, we only need to show that
$\lz(x_2,d(x_2,y))\ls \lz(x_1,d(x_1,y))$.
Notice that, for $x_1\in B$ and $y\in {\cx\backslash (kB)}$,
$$d(x_1,y)\geq d(y,c_B)-d(x_1,c_B)>2r_B-r_B=r_B.$$
It then follows that
$$d(x_2,y)\leq d(x_2,x_1)+d(x_1,y)
<2\,r_B+d(x_1,y)\leq3\,d(x_1,y),$$
which, together with \eqref{2.2}, the assumption that
$d(x_1,y)\leq d(x_2,y)$
and \eqref{2.1}, implies that
$$\lz(x_2,d(x_2,y))\sim \lz(y,d(x_2,y))
\ls\lz(x_1,d(x_2,y))\ls\lz\lf(x_1,3\,d(x_1,y)\r)
\ls \lz(x_1,d(x_1,y)).$$
This finishes the proof of Lemma \ref{l2.3}.
\end{proof}

The following definition of geometrically  doubling is well
known in analysis on metric spaces, which can be found in
Coifman and Weiss \cite[pp.\,66-67]{cw71},
and is also known as \emph{metrically doubling}
(see, for example, \cite[p.\,81]{he}).
Moreover, spaces of homogeneous type
are geometrically doubling, which was proved
by Coifman and Weiss in \cite[pp.\,66-68]{cw71}.

\begin{definition}\label{d2.4}
A metric space $(\cx,d)$ is said to be {\it geometrically doubling}
if there exists some $N_0\in \nn$ such that, for any ball
$B(x,r)\st \cx$ with $x\in\cx$ and $r\in(0,\fz)$,
there exists a finite ball covering $\{B(x_i,r/2)\}_i$ of
$B(x,r)$ such that the cardinality of this covering is at most $N_0$.
\end{definition}

\begin{remark}\label{r2.5}
Let $(\cx,d)$ be a metric space. In \cite{h10}, Hyt\"onen showed
that the following statements are mutually equivalent:
\vspace{-0.25cm}
\begin{itemize}
  \item[\rm(i)] $(\cx,d)$ is geometrically doubling;
\vspace{-0.25cm}
  \item[\rm(ii)] for any $\ez\in (0,1)$ and any ball
  $B(x,r)\st \cx$
with $x\in\cx$ and $r\in(0,\fz)$,
there exists a finite ball covering $\{B(x_i,\ez r)\}_i$ of
$B(x,r)$ such that the cardinality of this covering is
at most $N_0\ez^{-n_0}$, here and hereafter, $N_0$ is
as in Definition \ref{d2.4} and
$n_0:=\log_2N_0$;
\vspace{-0.25cm}
  \item[\rm(iii)] for every $\ez\in (0,1)$, any ball
  $B(x,r)\st \cx$
with $x\in\cx$ and $r\in(0,\fz)$ contains
at most $N_0\ez^{-n_0}$ centers of disjoint balls
$\{B(x_i,\ez r)\}_i$;
\vspace{-0.25cm}
  \item[\rm(iv)] there exists $M\in \nn$ such that
  any ball $B(x,r)\st \cx$
with $x\in\cx$ and $r\in(0,\fz)$ contains at most
$M$ centers $\{x_i\}_i$ of
  disjoint balls $\{B(x_i, r/4)\}_{i=1}^M$.
  \end{itemize}
\end{remark}

A metric measure space $(\cx,d,\mu)$ is called  a
\emph{metric measure space of non-homogeneous type}
if $(\cx,d)$ is geometrically doubling and $(\cx,d,\mu)$
is upper doubling. Based on Remark \ref{r2.2}(ii), from now on,
we \emph{always assume} that $(\cx,d,\mu)$ is a metric measure
space of non-homogeneous type with the dominating function $\lz$
satisfying \eqref{2.2} and, for any two balls $B, S\subset\cx$,
if $B=S$, then $c_B=c_S$ and $r_B=r_S$; see \cite[pp.\,314-315]
{flyy15} for some details.

Although the measure doubling condition is not assumed uniformly
for all balls in the metric measure space $(\cx,d,\mu)$ of
non-homogeneous type, it was shown in \cite{h10} that there still
exist many balls which have the following $(\az,\bz)$-doubling
property.

\begin{definition}\label{d2.6}
Let $\az,\,\bz\in (1,\fz)$. A ball $B\st \cx$ is said to be
\emph{$(\az,\bz)$-doubling}
if $\mu(\az B)\le \bz\mu(B)$.
\end{definition}

To be precise, it was proved in \cite[Lemma 3.2]{h10} that,
if a metric measure space $(\cx,d,\mu)$ is upper doubling
and $\az,\,\bz\in(1,\fz)$ with
$\bz>[C_{(\lz)}]^{\log_2\az}=:\az^\nu$, then, for any ball
$B\st \cx$, there exists some $j\in \zz_+$ such that
$\az^jB$ is $(\az,\bz)$-doubling. Moreover, let $(\cx,d)$
be geometrically doubling, $\bz>\az^{n_0}$ with
$n_0:=\log_2N_0$ and $\mu$ a Borel measure on $\cx$ which
is finite on bounded sets. Hyt\"onen \cite[Lemma 3.3]{h10}
also showed that, for $\mu$-almost every $x\in \cx$, there
exist arbitrary small $(\az,\bz)$-doubling balls centered
at $x$. Furthermore, the radii of these balls may be chosen
to be of the form $\az^{-j}r$ for $j\in\nn$ and any preassigned
number $r\in(0, \fz)$. Throughout this article,
for any $\az\in (1,\fz)$ and ball $B$, the \emph{smallest
$(\az,\bz_\az)$-doubling ball of the form $\az^j B$ with
$j\in \zz_+$} is denoted by $\wz B^\az$, where
\begin{equation}\label{2.3}
\bz_\az:=\az^{3(\max\{n_0,\,\nu\})}+[\max\{5\az,\,30\}]^{n_0}
+[\max\{3\az,30\}]^\nu.
\end{equation}
Also, for any ball $B$ of $\cx$, we denote by $\wz{B}$ the
smallest $(6,\bz_{6})$-doubling cube of the form $6^jB$ with
$j\in{\mathbb{Z}}_+$, especially, throughout this paper.

The following discrete coefficient $\kbsp$  was first introduced
by Bui and Duong \cite{ad13} as analogous of the quantity
introduced by Tolsa \cite{t01} (see also \cite{t012, t0302})
in the setting of non-doubling measures; see also
\cite{fyy14j, flyy15}.

\begin{definition}\label{d2.7}
For any $\rho\in(1,\fz)$ and any two balls
$B\subset S\subset\cx$, let
\begin{eqnarray*}
\kbsp:=1+\sum_{k=-\lfloor \log_{\rho}2\rfloor}^
{N_{B,S}^{(\rho)}}
\frac {\mu({\rho}^kB)} {\lz(c_B,{\rho}^kr_B)},
\end{eqnarray*}
where $N_{B,S}^{(\rho)}$ is the \emph{smallest integer} satisfying
$\rho^{N_{B,S}^{(\rho)}}r_B\geq r_S$.
\end{definition}
\begin{remark}
$(\rm{i})$ By a change of variables and \eqref{2.1}, we easily
conclude that
\begin{eqnarray*}
\kbsp\sim 1+\sum_{k=1}^{N_{B,S}^{(\rho)}+
\lfloor \log_{\rho}2\rfloor+1}\frac {\mu({\rho}^kB)}
{\lz(c_B,{\rho}^kr_B)},
\end{eqnarray*}
where the implicit equivalent positive constants are
independent of balls $B\subset S\subset\cx$, but
depend on $\rho$.

$(\rm{ii})$ A continuous version, $K_{B,S}$, of the
coefficient in Definition \ref{d2.7} was introduced in
\cite{h10} and \cite{hyy12} as follows. For any two
balls $B\subset S\subset\cx$, let
\begin{eqnarray*}
K_{B,S}:=1+\int_{(2S)\backslash B}
\frac1 {\lz(c_B,d(x,c_B))}\,d\mu(x).
\end{eqnarray*}
It was proved in \cite{hyy12} that $K_{B,S}$ has all properties
similar to those for $\kbsp$ as in Lemma \ref{l2.9} below.
Unfortunately, $K_{B,S}$ and $\kbsp$ are usually not equivalent,
but, for $(\mathbb{R}^d,|\cdot|,\mu)$ with $\mu$ as in \eqref{1.2},
\begin{equation}\label{2.4}
K_{B,S}\sim \kbsp
\end{equation}
with implicit equivalent positive constants independent of
$B$ and $S$; see \cite{fyy14j} for more details on this.
\end{remark}

The following useful properties of $\kbsp$ were proved in
\cite{flyy15}.

\begin{lemma}\label{l2.9}
Let $(\cx, d, \mu)$ be a metric measure space of non-homogeneous
type.

{\rm (i)} For any $\rho\in(1,\fz)$, there exists a positive
constant $C_{(\rho)}$, depending on $\rho$, such that, for
all balls $B\subset R\subset S$, $\kbrp\leq C_{(\rho)}\kbsp$.

{\rm (ii)} For any $\az\in[1,\fz)$ and $\rho\in(1,\fz)$,
there exists a positive constant $C_{(\az,\,\rho)}$,
depending on $\az$ and $\rho$, such that, for all balls
$B\subset S$ with $r_S\leq \az r_B$, $\kbsp\leq C_{(\az,\,\rho)}$.

{\rm (iii)} For any $\rho\in(1,\fz)$, there exists a positive
constant $C_{(\rho,\,\nu)}$, depending on $\rho$ and $\nu$,
such that, for all balls
$B$, $\wz K_{B,\wz{B}^{\rho}}^{(\rho)}\leq C_{(\rho,\,\nu)}$.
Moreover, letting $\az,\,\bz\in(1,\fz)$, $B\subset S$ be any two
concentric balls such that there exists no $(\az,\bz)$-doubling
ball in the form of $\az^kB$
with $k\in\mathbb{N}$, satisfying $B\subset \az^kB\subset S$,
then there exists a positive constant $C_{(\az,\,\bz,\,\nu)}$,
depending on $\az,\,\bz$ and $\nu$, such that
$\kbsp\leq C_{(\az,\,\bz,\,\nu)}$.

{\rm (iv)} For any $\rho\in(1,\fz)$, there exists a positive
constant $c_{(\rho,\,\nu)}$, depending on $\rho$ and $\nu$,
such that, for all balls $B\subset R\subset S$,
\begin{eqnarray*}
\kbsp\leq \kbrp+c_{(\rho,\,\nu)}\wz K_{R,S}^{(\rho)}.
\end{eqnarray*}

{\rm (v)} For any $\rho\in(1,\fz)$, there exists a positive
constant $\wz{c}_{(\rho,\,\nu)}$, depending on $\rho$ and
$\nu$, such that, for all balls
$B\subset R\subset S$,
$\wz K_{R,S}^{(\rho)}\leq \wz{c}_{(\rho,\,\nu)}\kbsp$.
\end{lemma}

\begin{lemma}\label{l2.10}
Let $(\cx,d,\mu)$ be a metric measure space of non-homogeneous
type and $\rho_1,\,\rho_2\in(1,\fz)$.
Then there exist positive constants $c_{(\rho_1,\rho_2,\nu)}$
and $C_{(\rho_1,\rho_2,\nu)}$, depending on $\rho_1,\rho_2$
and $\nu$, such that, for all balls $B\st S$,
\begin{eqnarray*}
c_{(\rho_1,\rho_2,\nu)}\wz K_{B,S}^{(\rho_1)}
\leq \wz K_{B,S}^{(\rho_2)}
\leq C_{(\rho_1,\rho_2,\nu)}\wz K_{B,S}^{(\rho_1)}.
\end{eqnarray*}
\end{lemma}

Now we recall the atomic Hardy space
$\wz H_{\rm{atb},\rho}^{1,q,\gz}(\mu)$
and its dual space
$\mathop\mathrm {\widetilde{RBMO}_{\rho,\,\gz}}(\mu)$
associated with $\kbsp$,
which were first introduced by Fu et al. \cite{fyy14j}.

\begin{definition}\label{d2.11}
Let $\rho\in(1,\fz)$, $q\in(1,\fz]$ and $\gz\in [1,\fz)$.
A function $b\in L^1(\mu)$ is called a
$(q,\gz,\rho)_\lz$-atomic block if

{\rm (i)} there exists a ball $B$ such that $\supp b\subset B$;

{\rm (ii)} $\int_{\cx} b(x)\,d\mu(x)=0$;

{\rm (iii)} for any $j\in\{1,2\}$, there exist a function $a_j$
supported on a ball $B_j\subset B$ and a number
$\lz_j\in \mathbb{C}$ such that $b=\lz_1a_1+\lz_2a_2$ and
\begin{eqnarray*}
\|a_j\|_{L^q(\mu)}\leq [\mu(\rho B_j)]^{1/q-1}
\lf[\wz K_{B_j,B}^{(\rho)}\r]^{-\gz}.
\end{eqnarray*}
Moreover, let
\begin{eqnarray*}
|b|_{\wz H_{\rm{atb},\rho}^{1,q,\gz}(\mu)}:=|\lz_1|+|\lz_2|.
\end{eqnarray*}
\end{definition}

A function $f\in L^1(\mu)$ is said to belong to the atomic
Hardy space $\wz H_{\rm{atb},\rho}^{1,q,\gz}(\mu)$
if there exist $(q,\gz, \rho)_\lz$-atomic blocks
$\{b_i\}_{i=1}^{\fz}$
such that $f=\sum_{i=1}^{\fz}b_i$ in $L^1(\mu)$ and
\begin{eqnarray*}
\sum_{i=1}^{\fz} |b_i|_{\wz H_{\rm{atb},\rho}^{1,q,\gz}(\mu)}<\fz.
\end{eqnarray*}
The $\wz H_{\rm{atb},\rho}^{1,q,\gz}(\mu)$ norm of $f$ is
defined by
\begin{eqnarray*}
\|f\|_{\wz H_{\rm{atb},\rho}^{1,q,\gz}(\mu)}:=
\inf\lf\{\sum_{i=1}^{\fz}
|b_i|_{\wz H_{\rm{atb},\rho}^{1,q,\gz}(\mu)}\r\},
\end{eqnarray*}
where the infimum is taken over all the possible
decompositions of $f$ as above.

\begin{remark}
{\rm (i)} When $(\cx,d,\mu)=(\mathbb{R}^d,|\cdot|,\mu)$
with $\mu$ as in \eqref{1.2},
by \eqref{2.4}, we see that
$\wz H_{\rm{atb},\rho}^{1,q,\gamma}(\mu)$
becomes the atomic Hardy space
$H_{\rm{atb},\rho}^{1,q,\gamma}(\mu)$ in \cite{t01}
for $\gamma=1$ and in \cite{hmy05} for $\gamma\in(1,\fz)$.
For general metric measure spaces of non-homogeneous type,
if we replace $\wz K_{B,S}^{(\rho)}$ by $K_{B,S}$ in
Definition \ref{d2.11}, then
$\wz H_{\rm{atb},\rho}^{1,q,\gamma}(\mu)$
becomes the atomic Hardy space
$H_{\rm{atb},\rho}^{1,q,\gamma}(\mu)$ in \cite{ad13,hyy12}.
Obviously, for $\rho\in(1,\fz)$, $q\in(1,\fz]$ and
$\gamma\in[1,\fz)$, we always have
\begin{eqnarray*}
\wz H_{\rm{atb},\rho}^{1,q,\gamma}(\mu)
\st H_{\rm{atb},\rho}^{1,q,\gamma}(\mu).
\end{eqnarray*}

{\rm (ii)} It was pointed out by Fu et al. \cite{fyy14j} that,
for each $q\in(1,\fz]$, the atomic Hardy space
$\wz H_{\rm{atb},\rho}^{1,q,\gamma}(\mu)$
is independent of the choices of $\rho$ and $\gamma$ and that,
for all $q\in(1,\fz)$, the spaces
$\wz H_{\rm{atb},\rho}^{1,q,\gamma}(\mu)$
and $\wz H_{\rm{atb},\rho}^{1,\fz,\gamma}(\mu)$ coincide
with equivalent norms. Thus, in what follows, we denote
$\wz H_{\rm{atb},\rho}^{1,q,\gamma}(\mu)$
simply by $\wz H^1(\mu)$.
\end{remark}

\begin{definition}\label{d2.13}
Let $\rho\in(1,\fz)$ and $\gz\in[1,\fz)$. A function
$f\in L_{loc}^1(\mu)$ is said to be in the space
$\mathop\mathrm {\widetilde{RBMO}_{\rho,\,\gz}}(\mu)$
if there exist
a positive constant $\wz{C}$ and, for any ball $B\subset \cx$,
a number $f_B$ such that
\begin{eqnarray}\label{2.5}
\frac1 {\mu(\rho B)} \int_B |f(x)-f_B|\,d\mu(x)\leq \wz{C}
\end{eqnarray}
and, for any two balls $B$ and $B_1$ such that $B\subset B_1$,
\begin{eqnarray}\label{2.6}
|f_B-f_{B_1}|\leq \wz{C}\lf[\wz K_{B,B_1}^{(\rho)}\r]^{\gz}.
\end{eqnarray}
The infimum of the positive constant $\wz{C}$ satisfying
both \eqref{2.5} and \eqref{2.6} is defined to be the
$\rm{\widetilde{RBMO}_{\rho,\,\gz}}(\mu)$
norm of $f$ and denoted by
$\|f\|_{\rm{\widetilde{RBMO}_{\rho,\,\gz}}(\mu)}$.
\end{definition}

\begin{remark}
{\rm (i)} It was pointed out by Fu et al. \cite{fyy14j} that
the space $\rm{\widetilde{RBMO}_{\rho,\gz}}(\mu)$ is independent
of $\rho\in(1,\fz)$ and $\gz\in[1,\fz)$. In what follows,
we denote $\rm{\widetilde{RBMO}_{\rho,\gz}}(\mu)$
simply by $\rm{\widetilde{RBMO}}(\mu)$.

{\rm (ii)} When $(\cx,d,\mu)=(\mathbb{R}^d,|\cdot|,\mu)$
with $\mu$ as in \eqref{1.2}, by \eqref{2.4}, we see that
$\rbmo$ becomes the regularized $\rm{BMO(\mu)}$ space,
$\rm{RBMO(\mu)}$, introduced in \cite{t01} for $\gz=1$
and in \cite{hmy05} for $\gz\in(1,\fz)$. For general metric
measure spaces of non-homogeneous type, if we replace
$\kbsp$ by $K_{B,S}$ in Definition \ref{d2.13}, then
$\rm{\widetilde{RBMO}}(\mu)$ becomes the space
$\rm{RBMO}(\mu)$ in \cite{h10}. Obviously, for
$\rho\in(1,\fz)$ and $\gamma\in[1,\fz)$,
$\rm{RBMO}(\mu)\st \rbmo$. However, it is still
unclear whether we always have $\rm{RBMO}(\mu)= \rbmo$
or not.

{\rm (iii)} Let $\rho\in(1,\fz)$, $p\in(1,\fz]$ and
$\gz\in[1,\fz)$. It was pointed out by Fu et al.
\cite{fyy14j} that
$[\wz H_{\rm{atb},\rho}^{1,p,\gz}(\mu)]^*=\rbmo$.

\end{remark}

By some arguments similar to those used in the proofs of
\cite[Proposition 2.10]{hyy12} and \cite[Lemma 3.2]{ly11},
we obtain the following equivalent characterization of the
space $\rbmo$, the details being omitted.

\begin{lemma}\label{l2.15}
Let $\eta,\,\rho\in(1,\fz)$, and $\bz_\rho$ be as in
\eqref{2.4}. For $f\in L_{loc}^1(\mu)$,
the following statements are equivalent:

{\rm (i)} $f\in\rbmo$;

{\rm (ii)}  there exists a positive constant $C$ such that,
for all balls $B$,
\begin{equation}\label{2.7}
\frac1 {\mu(\eta B)}\int_B |f(x)-m_{\wz B^\rho}(f)|\,d\mu(x)\leq C
\end{equation}
and, for all $(\rho,\bz_\rho)$-doubling balls $B\st S$,
\begin{equation}\label{2.8}
|m_B(f)-m_S(f)|\leq C\kbsp.
\end{equation}
Moreover, the infimum of the above constant $C$ is
equivalent to $\|f\|_{\rbmo}$.
\end{lemma}

By an argument completely analogous to that used in the proof of
\cite[Proposition 6.1]{h10}, we obtain the following John-Nirenberg
inequality for $\rbmo$, the details being omitted.

\begin{proposition}\label{p2.16}
Let $(\cx,d,\mu)$ be a metric measure space of non-homogeneous type.
Then, for every $\rho\in(0,\fz)$, there exists a positive constant $c$
such that, for all $f\in\rbmo$, balls $B_0$ and $t\in(0,\,\fz)$,
\begin{equation*}
\mu(\{x\in B_0:\,|f(x)-f_{B_0}|>t\})
\leq 2\mu(\rho B_0)e^{-ct/\|f\|_{\rbmo}},
\end{equation*}
where $f_{B_0}$ is as in Definition \ref{d2.13}
with $B$ replaced by $B_0$.
\end{proposition}

\begin{corollary}\label{c2.17}
Let $(\cx,d,\mu)$ be a metric measure space of non-homogeneous type.
Then, for every $\rho\in(1,\fz)$ and $p\in[1,\fz)$,
there exists a constant $C$ such that, for all $f\in \rbmo$
and balls $B$,
\begin{eqnarray*}
\lf[\frac1 {\mu(\rho B)}
\int_B |f(x)-f_B|^p\,d\mu(x)\r]^{1/p}\leq C\|f\|_{\rbmo},
\end{eqnarray*}
where $f_B$ is as in Definition \ref{d2.13}.
\end{corollary}

At the end of this section, we establish the following
Calder\'{o}n-Zygmund decomposition analogous to
\cite[Theorem 6.3]{ad13} and its proof is also analogous
to that of \cite[Theorem 6.3]{ad13}, the details being
omitted. Let $\gz_0$ be a fixed positive constant satisfying
that $\gz_0>\max\{C_{(\lz)}^{3\,\log_2 6},6^{3n_0} \}$,
where $C_{(\lz)}$ is as in \eqref{2.1} and $n_0$ is as in
Remark \ref{r2.5}(ii).

\begin{lemma}\label{l2.18}
Let $p\in[1,\fz)$, $f\in L^p(\mu)$ and
$t\in(0,\fz)$ ($t>(\gz_0)^{1/p}\|f\|_
{L^p(\mu)}/{[\mu(\cx)]^{1/p}}$
when $\mu(\cx)<\fz$). Then the following hold true.

{\rm (i)} There exists an almost disjoint family $\{6B_j\}_j$ of
balls such that $\{B_j\}_j$ is pairwise disjoint,
$$\frac{1}{\mu(6^2B_j)}\int_{B_j}|f(x)|^p\,d\mu(x)
>\frac{t^p}{\gz_0} \quad{\text for\, all\, j,}$$
$$\frac{1}{\mu(6^2\eta B_j)}\int_{\eta B_j}|f(x)|^p\,d\mu(x)
\le\frac{t^p}{\gz_0} \quad{\text for\ all\ j\ and\ all\
\eta\in(2,\,\fz)},$$
and
$$
|f(x)|\le t \quad{\text for\ \mu-almost\ every\ x\in
\cx\setminus(\cup_j6B_j).}
$$

{\rm (ii)} For each j, let $S_j$ be a
$(3\times6^2,\,C_\lz^{\log_2(3\times 6^2)+1})$-doubling ball of
the family $\{(3\times 6^2)^kB_j\}_{k\in\nn}$ and
$\oz_j:=\chi_{6B_j}/(\sum_k\chi_{6B_k})$. Then there exists a
family $\{\vz_j\}_j$ of functions such that, for each $j$,
$\supp(\vz_j)\subset S_j$, $\vz_j$ has a constant sign on $S_j$,
$$\int_\cx\vz_j(x)\,d\mu(x)=\int_{6B_j}f(x)\oz_j(x)\,d\mu(x),$$
$$\sum_j|\vz_j(x)|\le \gz t \quad{\text for\ \mu-almost\ every\
x\in\cx,}$$
where $\gz$ is some positive constant, depending only on
$(\cx,\,\mu)$, and there exists a positive constant $C$,
independent of $f$, $t$ and $j$, such that, when $p=1$,
it holds true that
$$
\|\vz_j\|_{L^\fz (\mu)}\mu(S_j)\le
C\int_\cx|f(x)\oz_j(x)|\,d\mu(x)
$$
and, when $p\in(1,\,\fz)$, it holds true that
\begin{equation*}
\lf[\int_{S_j}|\vz_j(x)|^p\,d\mu(x)\r]^{1/p}[\mu(S_j)]^{1/p'}\le
\frac{C}{t^{p-1}}\int_\cx |f(x)\oz_j(x)|^p\,d\mu(x).
\end{equation*}

{\rm (iii)} For $p\in(1,\,\fz)$, if choosing $S_j$ in
{\rm(ii)} to be the smallest $(3\times6^2,\,C_\lz^{\log_2(3\times
6^2)+1})$-doubling ball of the family
$\{(3\times 6^2)^kB_j\}_{k\in\nn}$, then
$h:=\sum_j(f\oz_j-\vz_j)\in \wz H^1(\mu)$ and there exists
a positive constant $C$, independent of $f$ and $t$, such that
$$\|h\|_{\wz H^1(\mu)}\le
\frac{C}{t^{p-1}}\|f\|_{L^p(\mu)}^p.$$
\end{lemma}

\section{Boundedness of Commutators $T_b$ on $\wz H^1(\mu)$}\label{s3}
\hskip\parindent
In this section, we consider the boundedness from $\wz H^1(\mu)$
into $L^{1,\fz}(\mu)$ of the commutator generated by the $\rbmo$
function and the Calder\'on-Zygmund operator with kernel
satisfying only the size condition and some H\"ormander-type
condition.

To be precise, let $K$ be a $\mu$-locally integrable function
on $\{\cx \times\cx\}\setminus\{(x,x):\,x\in\cx\}$ satisfying the
\emph{size condition} that there exists a positive constant $C$
such that, for all $x,\,y\in\cx$ with $x\neq y$,
\begin{equation}\label{3.1}
|K(x,y)|\leq C\frac1{\lz(x,d(x,y))},
\end{equation}
and the \emph{H\"ormander-type condition} that there exists a
positive constant $C$ such that, for any $R\in(0,\fz)$ and
$y,\,y'\in\cx$ with $d(y,y')<R$,
\begin{equation}\label{3.2}
\sum_{l=1}^\fz l\int_{6^{l}R <d(x,y)\leq 6^{l+1}R}
\lf[|K(x,y)-K(x,{y}')|+
|K(y,x)-K({y}',x)|\r] \,d\mu(x)\leq C.
\end{equation}
A linear operator $T$ is called  a \emph{Calder\'on-Zygmund operator}
with kernel $K$ satisfying \eqref{3.1} and \eqref{3.2} if, for all
$f\in L^{\fz}_b(\mu):=\{f\in L^\fz(\mu):\,\supp(f)\ {\rm is\ bounded}\}$,
\begin{equation}\label{3.3}
Tf(x):=\int_{\cx}K(x,y)f(y)\,d\mu(y),\quad x\not\in\supp(f).
\end{equation}
Let $b\in \rm \rbmo$ and $T$ be a Calder\'{o}n-Zygmund operator defined above.
The commutator $T_b$, generated by $b$ and $T$, is defined by
setting, for any suitable function $f$,
\begin{equation}\label{3.4}
T_b f:=bTf-T(bf).
\end{equation}

Now we state the main result of this section as follows.

\begin{theorem}\label{t3.1}
Let $b\in\rm \rbmo$. Assume that the Calder\'on-Zygmund operator
$T$, defined by \eqref{3.3} associated with kernel $K$ satisfying
$\eqref{3.1}$ and $\eqref{3.2}$, is bounded on $L^2(\mu)$. Then the
commutator $T_b$ defined by \eqref{3.4} is bounded from $\wz H^1(\mu)$
into $L^{1,\infty}(\mu)$, that is, there exists a positive constant
$C$ such that, for all $t\in(0,\fz)$ and all functions $f\in \wz H^1(\mu)$,
\begin{eqnarray*}
\mu(\{x\in\cx:\,|T_bf(x)|>t\})
\leq C\|b\|_{\rbmo} t^{-1}\|f\|_{\wz H^1(\mu)}.
\end{eqnarray*}
\end{theorem}

To prove Theorem \ref{t3.1}, we need the following two lemmas.
\begin{lemma}\label{l3.2}
Let $T$ be a Calder\'on-Zygmund operator defined by \eqref{3.3}
associated with kernel $K$ satisfying $\eqref{3.1}$ and $\eqref{3.2}$.
Assume that $T$ is bounded on $L^2(\mu)$. Then
\begin{itemize}
\item[{\rm (i)}] $T$ is bounded from $L^1(\mu)$ into $L^{1,\fz}(\mu)$;

\item[{\rm (ii)}] $T$ is bounded on $L^p(\mu)$ for all $p\in(1,\fz)$.
\end{itemize}
\end{lemma}

\begin{proof}
The proof of (i) is similar to that of
$\rm{(i)}\Longrightarrow\rm{(ii)}$ of \cite[Theorem 1.6]{lmy14},
the details being omitted. By (i), together with the
Marcinkiewicz interpolation theorem and a standard duality, we then
obtain the desired result of (ii), which completes
the proof of Lemma \ref{l3.2}.
\end{proof}

The following generalized H\"older inequality is a special case of
\cite[Lemma 4.1]{fyy12} (see also \cite[pp.\,246-247]{hmy}
in the setting of $\rd$ with $\mu$ as in  \eqref{1.2} and
\cite[Lemmas 2.2 and 2.3]{pt02} for the setting of $\rd$
with $\mu$ being the $d$-dimensional Lebesgue measure).

\begin{lemma}\label{l3.3}
There exists a positive constant $C$ such that,
for all locally integrable functions $f$ and $g$,
and all balls $B$,
\begin{equation}\label{3.5}
\dfrac{1}{\mu(2B)}\dint_B|f(x)g(x)|\,d\mu(x)\le
C\|g\|_{\exp L(\mu),B}\|f\|_{L\log L(\mu),B},
\end{equation}
where
$$\|f\|_{L\log L(\mu),B}:=\inf\lf\{s\in(0,\fz):\,
\frac 1 {\mu(2B)}
\int_B {\frac {|f(x)|} s}\log
\lf(2+{\frac {|f(x)|} s}\r)\,d\mu(x)\leq1 \r\}$$
and
$$\|f\|_{\exp L(\mu),B}:=\inf\lf\{s\in(0,\fz):\,\frac 1 {\mu(2B)}
\int_B \exp\lf(\frac {|f(x)|} s\r)\,d\mu(x)\leq2\r\}.$$
\end{lemma}

Now we can show Theorem \ref{t3.1} as follows.

\begin{proof}[Proof of Theorem \ref{t3.1}]
For each fixed $f\in \wz H^1(\mu)$, by Definition \ref{d2.11},
we have a decomposition $f=\sum_{j=1}^\fz h_j$,
where, for any $j\in\nn$, $h_j$ is an $(\infty, 2,12)_\lz$-atomic
block, $\supp h_j\subset S_j$, $S_j$ is a ball of $\cx$, and
\begin{equation*}
\sum_{j=1}^\fz |h_j|_{\wz H_{\rm{atb},12}^{1,\infty,2}(\mu)}
\leq 2\|f\|_{\wz H^1 (\mu)}.
\end{equation*}
Moreover, for each fixed $j$, we can further decompose $h_j$ as
$ h_j=r_{j,1}a_{j,1}+r_{j,2}a_{j,2}$, where,
for any $i\in\{1,2\}$, $r_{j,i}\in\cc$, $a_{j,i}$
is a bounded function supported on some ball
$B_{j,i}\subset S_j$ satisfying
\begin{eqnarray*}
\|a_{j,i}\|_{L^\infty(\mu)}
\le\lf[\mu\lf(12B_{j,i}\r)
\lf\{\wz K_{B_{j,i},S_j}^{(12)}\r\}^2\r]^{-1}
\end{eqnarray*}
and $ |h_j|_{\wz H_{\rm{atb},12}^{1,\infty,2}
(\mu)}=|r_{j,1}|+|r_{j,2}|$.
By Lemma \ref{l2.10}, we further conclude that,
for any $j$ and any $i\in \{1,2\}$,
\begin{eqnarray}\label{3.6}
\|a_{j,i}\|_{L^\infty(\mu)}
\ls\lf[\mu\lf(12B_{j,i}\r)
\lf\{\wz K_{B_{j,i},S_j}^{(6)}\r\}^2\r]^{-1}.
\end{eqnarray}
Write
$$
T_bf=\sum_{j=1}^\fz \lf[b-m_{\wz S_j}(b)\r]Th_j
+T\lf(\sum_{j=1}^\fz \lf[m_{\wz S_j}(b)-b(\cdot)\r]h_j\r)
=:T_b^{\rm I}f+T_b^{\rm II}f.
$$

By Lemma \ref{l3.2}, we know that $T$ is bounded from
$L^1(\mu)$ into $L^{1,\infty}(\mu)$. It then follows that
\begin{eqnarray}\label{3.7}
\mu(\{x\in\cx:\ \lf|T_b^{\rm{II}}f(x)\r|>t\})
&&\ls\frac1 t\sum_{j=1}^\fz \int_{S_j}
\lf|b(x)-m_{\wz S_j}(b)\r|\lf|h_j(x)\r| \,d\mu(x)\\
&&\ls\frac1 t\sum_{j=1}^\fz\lf[|r_{j,1}|\int_{B_{j,1}}
\lf|b(x)-m_{\wz S_j}(b)\r||a_{j,1}(x)|\,d\mu(x)\r.\noz\\
&&\hs\hs+\lf.|r_{j,2}|\int_{B_{j,2}}
\lf|b(x)-m_{\wz S_j}(b)\r||a_{j,2}(x)|\,d\mu(x)\r]\noz\\
&&=:\frac1 t\sum_{j=1}^\fz({\rm E}+{\rm F}).\noz
\end{eqnarray}
By \eqref{2.8} and Lemma \ref{l2.9}, we have
$$
\lf|m_{\wz S_j}(b)-m_{\wz B_{j,1}}(b)\r|
\ls \wz K_{\wz B_{j,1},\wz S_j}^{(6)}\|b\|_{\rbmo}
\ls \wz K_{B_{j,1},S_j}^{(6)}\|b\|_{\rbmo},
$$
which, together with \eqref{2.7} and \eqref{3.6}, leads to
\begin{eqnarray*}
{\rm E}&&\ls|r_{j,1}|\|a_{j,1}\|_{L^\infty(\mu)}
\lf[\int_{B_{j,1}}
\lf|b(x)-m_{\wz B_{j,1}}(b)\r| \,d\mu(x)
+\lf|m_{\wz S_j}(b)-m_{\wz B_{j,1}}(b)\r|\mu(B_{j,1})\r]\\
&&\ls\|b\|_{\rbmo}|r_{j,1}|\lf[\mu(12B_{j,1})
\lf\{\wz K_{B_{j,1},S_j}^{(6)}\r\}^2\r]^{-1}
\lf[\mu(2B_{j,1})+\wz K_{B_{j,1},S_j}^{(6)}\mu(B_{j,1})\r]\\
&&\ls\|b\|_{\rbmo}|r_{j,1}|.
\end{eqnarray*}
Similarly,
\begin{eqnarray*}
{\rm F}\ls\|b\|_{\rbmo}|r_{j,2}|.
\end{eqnarray*}
Combining $\rm E$ and $\rm F$, we conclude that
\begin{equation}\label{3.8}
\mu(\{x\in\cx:\,|T_b^{\rm{II}}f(x)|>t\})
\ls\|b\|_{\rbmo}t^{-1}\|f\|_{\wz H^1(\mu)}.
\end{equation}

Now we turn to estimate $T_b^{\rm I}f$. Write
\begin{eqnarray*}
\mu(\{x\in\cx:\,|T_b^{\rm I}f(x)|>t\})
&&\leq t^{-1}\sum_{j=1}^\fz \int_{6S_j}
\lf|b(x)-m_{\wz{S_j}}(b)\r||Th_j(x)|\,d\mu(x)\\
&&\quad+t^{-1}\sum_{j=1}^\fz
\int_{\cx\backslash{6S_j}}
\cdots=:{\rm G}+{\rm H}.
\end{eqnarray*}

We first estimate $\rm G$. For each fixed $j$, write
\begin{eqnarray*}
\int_{6S_j}\lf|b(x)-m_{\wz{S_j}}(b)\r||Th_j(x)|\,d\mu(x)
&&\leq|r_{j,1}|\int_{6S_j}
\lf|b(x)-m_{\wz{S_j}}(b)\r|\lf|Ta_{j,1}(x)\r|\,d\mu(x)\\
&&\hs\hs+|r_{j,2}|\int_{6S_j}
\lf|b(x)-m_{\wz{S_j}}(b)\r|\lf|Ta_{j,2}(x)\r|\,d\mu(x)\\
&&=:{\rm L}_{j,1}+{\rm L}_{j,2}.
\end{eqnarray*}
Since the two terms ${\rm L}_{j,1}$ and ${\rm L}_{j,2}$ can
be estimated in a similar way, we only deal with
${\rm L}_{j,1}$. Write
\begin{eqnarray*}
{\rm L}_{j,1}&&\leq|r_{j,1}|\int_{6S_j\backslash{6B_{j,1}}}
\lf|b(x)-m_{\wz{S_j}}(b)\r|\lf|Ta_{j,1}(x)\r|\,d\mu(x)\\
&&\quad+|r_{j,1}|\int_{6B_{j,1}}
\lf|b(x)-m_{\wz{6B_{j,1}}}(b)\r|\lf|Ta_{j,1}(x)\r|\,d\mu(x)\\
&&\quad+|r_{j,1}|\lf|m_{\wz{6B_{j,1}}}(b)-m_{\wz{S_j}}(b)\r|
\int_{6B_{j,1}}\lf|Ta_{j,1}(x)\r|\,d\mu(x)=
:{\rm U}_j+{\rm V}_j+{\rm W}_j.
\end{eqnarray*}
The H\"{o}lder inequality, together with \eqref{2.8},
Lemma \ref{l2.9}, the boundedness of $T$ on $L^2(\mu)$ and
\eqref{3.6}, implies that
\begin{eqnarray*}
{\rm W}_j&&\ls\|b\|_{\rbmo}\lf|r_{j,1}\r|
\wz K_{\wz{6B_{j,1}},\wz{S_j}}^{(6)}
\lf\|Ta_{j,1}\r\|_{L^2(\mu)}
\lf[\mu\lf(6B_{j,1}\r)\r]^{\frac1 2}\\
&&\ls\|b\|_{\rbmo}\lf|r_{j,1}\r| \wz K_{B_{j,1},S_j}^{(6)}
\lf\|a_{j,1}\r\|_{L^2(\mu)}
\lf[\mu\lf(6B_{j,1}\r)\r]^{\frac1 2}\\
&&\ls\|b\|_{\rbmo}\lf|r_{j,1}\r| \wz K_{B_{j,1},S_j}^{(6)}
\mu\lf(12B_{j,1}\r)
\lf\|a_{j,1}\r\|_{L^\fz(\mu)}
\ls\|b\|_{\rbmo}\lf|r_{j,1}\r|.
\end{eqnarray*}
From the H\"older inequality, Corollary \ref{c2.17},
the boundedness of $T$ on $L^2(\mu)$ and \eqref{3.6},
we deduce that
\begin{eqnarray*}
{\rm V}_j&&\ls\lf|r_{j,1}\r|\lf[\int_{6B_{j,1}}
\lf|b(x)-m_{\wz{6B_{j,1}}}(b)\r|^2\,d\mu(x)\r]^{\frac1 2}
\lf\|Ta_{j,1}\r\|_{L^2(\mu)}\\
&&\ls\|b\|_{\rbmo}\lf|r_{j,1}\r|
\lf[\mu(12B_{j,1})\r]^{\frac1 2}\|a_{j,1}\|_{L^2(\mu)}
\ls\|b\|_{\rbmo}\lf|r_{j,1}\r|.
\end{eqnarray*}
To estimate ${\rm U}_j$, we first observe that,
for $x\notin 6B_{j,1}$
and $y\in B_{j,1}$, $d(x,y)\sim d(x,c_{B_{j,1}})$.
It then follows from \eqref{3.1} that, for $x\notin 6B_{j,1}$,
\begin{eqnarray*}
\lf|Ta_{j,1}(x)\r|&&=\lf|\int_{\cx} K(x,y)a_{j,1}(y)\,d\mu(y)\r|
\leq \int_{B_{j,1}} \lf|K(x,y)\r||a_{j,1}(y)|\,d\mu(y)\\
&&\ls\frac{\|a_{j,1}\|_{L^1(\mu)}} {\lz(x,d(x,c_{B_{j,1}}))}
\ls\frac{\|a_{j,1}\|_{L^\fz(\mu)}\mu(B_{j,1})}
{\lz(x,d(x,c_{B_{j,1}}))}.
\end{eqnarray*}
Let $N_1:=N_{6B_{j,1},6S_j}^{(6)}+\lfloor \log_{6}2\rfloor+1$.
A straightforward computation, via the above estimate,
\eqref{2.7}, \eqref{2.8}, Lemma \ref{l2.9}, \eqref{2.2}
and \eqref{3.6}, shows that
\begin{eqnarray*}
{\rm U}_j
&&\ls|r_{j,1}|\|a_{j,1}\|_{L^\fz(\mu)}\mu(B_{j,1})
\int_{6S_j\backslash{6B_{j,1}}} \frac{|b(x)-m_
{\wz{S_j}}(b)|} {\lz(x,d(x,c_{B_{j,1}}))}\,d\mu(x)\\
&&\ls|r_{j,1}|\|a_{j,1}\|_{L^\fz(\mu)}\mu(B_{j,1})
\sum_{k=1}^{N_1} \lf[\int_{6^{k+1}B_{j,1}
\backslash{6^kB_{j,1}}}\frac{|b(x)-m_{\wz{6^{k+1}
B_{j,1}}}(b)|}{\lz(x,d(x,c_{B_{j,1}}))}\,d\mu(x)\r.\\
&&\lf.\quad+\lf|m_{\wz{6^{k+1}B_{j,1}}}(b)-m_{\wz{S_j}}(b)\r|
\int_{6^{k+1}B_{j,1}\backslash{6^kB_{j,1}}}
\frac1{\lz(x,d(x,c_{B_{j,1}}))}\,d\mu(x)\r]\\
&&\ls\|b\|_{\rbmo}|r_{j,1}|\|a_{j,1}\|_{L^{\fz}(\mu)}
\mu(B_{j,1})\\
&&\hs\hs\times\sum_{k=1}^{N_1}
\lf[\frac{\mu(6^{k+2}B_{j,1})} {\lz(c_{B_{j,1}},6^{k+2}
r_{B_{j,1}})}+\wz K_{B_{j,1},S_j}^{(6)}\frac{\mu(6^{k+1}
B_{j,1})} {\lz(c_{B_{j,1}},6^{k+1}r_{B_{j,1}})}\r]\\
&&\ls\|b\|_{\rbmo}|r_{j,1}|\|a_{j,1}\|_{L^{\fz}(\mu)}
\mu(B_{j,1})\wz K_{B_{j,1},S_j}^{(6)}\lf[1+\sum_{k=1}^{N_1}
\frac{\mu(6^kB_{j,1})} {\lz(c_{B_{j,1}},6^kr_{B_{j,1}})}\r]\\
&&\ls\|b\|_{\rbmo}|r_{j,1}|\|a_{j,1}\|_{L^{\fz}(\mu)}
\mu(B_{j,1})\lf[\wz K_{B_{j,1},S_j}^{(6)}\r]^2
\ls\|b\|_{\rbmo}|r_{j,1}|.
\end{eqnarray*}
Combining the estimates for ${\rm U}_j$, ${\rm V}_j$
and ${\rm W}_j$, we obtain
$$ {\rm L}_{j,1}\ls\|b\|_{\rbmo}\lf|r_{j,1}\r|,$$
which further implies that
$$ {\rm G}\ls t^{-1}\|b\|_{\rbmo}\sum_{j=1}^\fz
|h_j|_{\wz H_{\rm{atb},12}^{1,\infty,2}(\mu)}
\ls t^{-1}\|b\|_{\rbmo}\|f\|_{\wz H^1(\mu)}.$$

It remains to estimate ${\rm H}$. The vanishing moment of $h_j$,
together with the Fubini theorem, implies that
\begin{eqnarray*}
{\rm H}&&=t^{-1}\sum_{j=1}^\fz \int_{\cx\backslash{6S_j}}
\lf|b(x)-m_{\wz{S_j}}(b)\r|\lf|
\int_{S_j}\lf[K(x,y)-K(x,c_{S_j})\r]h_j(y)\,d\mu(y)\r|\,d\mu(x)\\
&&\leq t^{-1}\sum_{j=1}^\fz \int_{S_j}|h_j(y)|
\int_{\cx\backslash{6S_j}}\lf|b(x)-m_{\wz{S_j}}(b)\r|
\lf|K(x,y)-K(x,c_{S_j})\r|\,d\mu(x)\,d\mu(y).
\end{eqnarray*}
For each fixed $j$, write
\begin{eqnarray*}
&&\int_{\cx\backslash{6S_j}}\lf|b(x)-m_{\wz{S_j}}(b)\r|
\lf|K(x,y)-K(x,c_{S_j})\r|\,d\mu(x)\\
&&\hs\leq\sum_{k=1}^\fz
\int_{6^{k+1}S_j\backslash6^kS_j}
\lf|b(x)-m_{\wz{6^{k+1}S_j}}(b)\r|
\lf|K(x,y)-K(x,c_{S_j})\r|\,d\mu(x)\noz\\
&&\hs\hs+\sum_{k=1}^\fz \int_{6^{k+1}S_j\backslash6^kS_j}
\lf|m_{\wz{6^{k+1}S_j}}(b)-m_{\wz{S_j}}(b)\r|
|K(x,y)-K(x,c_{S_j})|\,d\mu(x)\noz\\
&&\hs=:{\rm H}_1+{\rm H}_2.
\end{eqnarray*}
By \eqref{2.8} and Lemma \ref{l2.9}, we have
\begin{eqnarray}\label{3.9}
\lf|m_{\wz{6^{k+1}S_j}}(b)-m_{\wz{S_j}}(b)\r|
&&\ls\wz K_{\wz{S_j},\wz{6^{k+1}S_j}}^{(6)}\|b\|_{\rbmo}\\
&&\ls \wz K_{S_j,6^{k+1}S_j}^{(6)}\|b\|_{\rbmo}
\ls k\|b\|_{\rbmo},\noz
\end{eqnarray}
which, together with \eqref{3.2}, implies that,
for any $y\in S_j$,
\begin{eqnarray*}
{\rm H}_2\ls\|b\|_{\rbmo}\sum_{k=1}^\fz
k\int_{6^{k+1}S_j\backslash6^kS_j}
|K(x,y)-K(x,c_{S_j})|\,d\mu(x)\ls\|b\|_{\rbmo}.
\end{eqnarray*}
For ${\rm H}_1$, from \eqref{3.5} and Proposition \ref{p2.16},
we deduce that
\begin{eqnarray*}
{\rm H}_1&&\ls\sum_{k=1}^\fz \mu(2\cdot6^{k+1}S_j)
\lf\|b-m_{\wz{6^{k+1}S_j}}(b)\r\|_
{\exp L(\mu),6^{k+1}S_j}\\
&&\hs\hs\times\lf\|\lf[K(\cdot,y)-K(\cdot,c_{S_j})\r]
\chi_{{6^{k+1}S_j}\backslash{6^kS_j}}\r\|_
{L\log L(\mu),6^{k+1}S_j}\\
&&\ls\|b\|_{\rbmo}\sum_{k=1}^\fz \mu(2\cdot6^{k+1}S_j)
\lf\|\lf[K(\cdot,y)-K(\cdot,c_{S_j})\r]
\chi_{{6^{k+1}S_j}\backslash{6^kS_j}}\r\|_
{L\log L(\mu),6^{k+1}S_j}.
\end{eqnarray*}
Choose
\begin{eqnarray*}
l_k:=\lf[\mu\lf(2\cdot6^{k+1}S_j\r)\r]^{-1}
\lf[k\int_{{6^{k+1}S_j}\backslash{6^kS_j}}
\lf|K(x,y)-K(x,c_{S_j})\r|d\mu(x)+2^{-k}\r].
\end{eqnarray*}
By \eqref{3.1}, \eqref{2.2} and Lemma \ref{l2.3},
we conclude that, for any $y\in S_j$,
\begin{eqnarray*}
&&\frac 1 {\mu\lf(2\cdot6^{k+1}S_j\r)}
\int_{{6^{k+1}S_j}\backslash{6^kS_j}}
\frac {|K(x,y)-K(x,c_{S_j})|} {l_k}
\log\lf(2+{\frac {|K(x,y)-K(x,c_{S_j})|}
{l_k}}\r)\,d\mu(x)\\
&&\hs\ls \frac1 {\mu\lf(2\cdot6^{k+1}S_j\r)}
\int_{{6^{k+1}S_j}\backslash{6^kS_j}}
\frac {|K(x,y)-K(x,c_{S_j})|} {l_k}\\
&&\hs\hs \times \log\lf(2+{\frac 1 {l_k\lz(x,d(x,y))}}
+{\frac 1 {l_k\lz(x,d(x,c_{S_j}))}}\r)\,d\mu(x)\\
&&\hs\ls \frac1 {\mu\lf(2\cdot6^{k+1}S_j\r)}
\int_{{6^{k+1}S_j}\backslash{6^kS_j}}
\frac {|K(x,y)-K(x,c_{S_j})|} {l_k}
\log\lf(2+\frac1 {l_k\lz(c_{S_j},d(x,c_{S_j}))}\r)
\,d\mu(x)\\
&&\hs\ls \frac1 {\mu\lf(2\cdot6^{k+1}S_j\r)}
\log\lf(2+\frac{2^k\mu(2\cdot6^{k+1}S_j)}
{\lz(c_{S_j},6^kr_{S_j})}\r)
\int_{{6^{k+1}S_j}\backslash{6^kS_j}}
\frac {|K(x,y)-K(x,c_{S_j})|} {l_k}\,d\mu(x)\\
&&\hs\ls \frac {k} {\mu\lf(2\cdot6^{k+1}S_j\r)
}\int_{{6^{k+1}S_j}\backslash{6^kS_j}}
\frac {|K(x,y)-K(x,c_{S_j})|} {l_k}\,d\mu(x)\ls 1,
\end{eqnarray*}
which implies that
\begin{eqnarray*}
\lf\|\lf\{K(\cdot,y)-K(\cdot,c_{S_j})\r\}
\chi_{{6^{k+1}S_j}\backslash{6^kS_j}}\r\|_{L\log
L(\mu),6^{k+1}S_j}
\ls l_k.
\end{eqnarray*}
From this and \eqref{3.2}, it follows that
\begin{eqnarray*}
{\rm H}_1&&\ls\|b\|_{\rbmo}\sum_{k=1}^\fz
\mu\lf(2\cdot6^{k+1}S_j\r)l_k\\
&&\ls\|b\|_{\rbmo}\sum_{k=1}^\fz
\lf[k\int_{{6^{k+1}S_j}\backslash{6^kS_j}}
|K(x,y)-K(x,c_{S_j})|\,d\mu(x)+2^{-k}\r]\\
&&\ls\|b\|_{\rbmo}.
\end{eqnarray*}
Combining the estimates for $\rm H_1$ and $\rm H_2$,
we then obtain
$${\rm H}\ls t^{-1}\|b\|_{\rbmo}\sum_{j=1}^\fz
\|h_j\|_{L^1(\mu)}
\ls t^{-1}\|b\|_{\rbmo}\|f\|_{\wz H^1(\mu)}.$$

We finally conclude that
\begin{eqnarray*}
\mu\lf(\{x\in\cx:\,|T_b^{\rm I}f(x)|>t \}\r)
\ls\|b\|_{\rbmo}t^{-1}\|f\|_{\wz H^1(\mu)},
\end{eqnarray*}
which, together with the estimate $\eqref{3.8}$,
completes the proof of Theorem \ref{t3.1}.
\end{proof}

\begin{remark}\label{r3.4}
Let $b\in\mathop\mathrm{RBMO}(\mu)$. Fu et al.
\cite[Theorem 3.10]{fyy12} obtained the boundedness
on Lebesgue spaces $L^p(\mu)$ with $p\in(1,\fz)$ of
the commutator $T_b$ generated by $b$ and $T$ with
kernel satisfying \eqref{3.1} and the following
\emph{stronger regularity condition}, that is, there exist
positive constants $C$, $\dz\in(0,1]$ and $c_{(K)}$,
depending on $K$, such that, for all $x,\ \wz{x},\ y\in\cx$
with $d(x,y)\geq c_{(K)}d(x,\wz{x})$,
\begin{equation}\label{3.10}
\lf|K(x,y)-K(\wz{x},y)\r|+\lf|K(y,x)-K(y,\wz{x})\r|\leq C\frac{[d(x,\wz{x})]^\delta}{[d(x,y)]^\delta\lz(x,d(x,y))}.
\end{equation}
A new example of  the operator with kernel satisfying
\eqref{3.1} and \eqref{3.10} is the so-called
Bergman-type operator appearing in \cite{vw};
see also \cite{hm12} for an explanation.
Notice that $\mathop\mathrm{RBMO}(\mu)\subset\rbmo$.
Theorem \ref{t3.1} also holds true for the commutator $T_b$
generated by $b\in\mathop\mathrm{RBMO}(\mu)$ and $T$
with kernel satisfying \eqref{3.1} and \eqref{3.10}.
Moreover, when $b\in\rbmo$ and $T$ with kernel satisfying
\eqref{3.1} and \eqref{3.2}, we also prove that the commutator
$T_b$ is bounded on $L^p(\mu)$ for all $p\in(1,\fz)$, which
improves \cite[Theorem 3.10]{fyy12}; see Section \ref{s5} below.
\end{remark}

\section{Boundedness of Commutators $T_{\az,b}$ on
$\wz H^1(\mu)$}\label{s4}
\hskip\parindent
In this section, we establish the boundedness
from $\wz H^1(\mu)$ into $L^{1/(1-\az),\fz}(\mu)$
of the commutator generated by the generalized
fractional integral $T_\alpha\,(\alpha\in(0,1))$
and the $\rm{\widetilde{RBMO}(\mu)}$ function.
We begin with the definition of the generalized
fractional integral.
\begin{definition}
Let $\az\in(0,1)$. A function
$K_\az\in L_{{\rm loc}}^1(\{\cx\times\cx\}\backslash
{\{(x,x):\,x\in\cx\}})$
is called a \emph{generalized fractional integral kernel}
if there exists a positive constant $C_{(K_\az)}$,
depending only on $K_\az$, such that

(i) for all $x,\,y\in\cx$ with $x\not=y$,
\begin{eqnarray}\label{4.1}
\lf|K_\az(x,y)\r|\leq C_{(K_\az)}\frac1
{[\lz(x,d(x,y))]^{1-\az}};
\end{eqnarray}

(ii) there exist positive constant $\delta\in(0,1]$
and $c_{(K_\az)}\in(0,\fz)$, depending only on $K_\az$,
such that, for all $x,\,\wz{x},\,y\in\cx$ with
$d(x,y)\geq c_{(K_\az)}d(x,\wz{x})$,
\begin{equation}\label{4.2}
\lf|K_\az(x,y)-K_\az(\wz{x},y)\r|+\lf|K_\az(y,x)-
K_\az(y,\wz{x})\r|
\leq C_{(K_\az)}{\frac{[d(x,\wz{x})]^\delta}
{[d(x,y)]^\delta[\lz(x,d(x,y))]^{1-\az}}}.
\end{equation}

A linear operator $T_\az$ is called a
\emph{generalized fractional integral}
with kernel $K_\az$ satisfying $\eqref{4.1}$ and
$\eqref{4.2}$ if, for all $f\in L_b^\fz(\mu)$
and $x\notin\supp f$,
\begin{eqnarray}\label{4.3}
T_{\az}f(x):=\int_\cx K_\az(x,y)f(y)\,d\mu(y).
\end{eqnarray}
\end{definition}
Let $b\in\rm \rbmo$ and $T_\az$ be the generalized
fractional integral. The the commutator $T_{\az,b}$,
generated by $b$ and $T_\az$, is defined by setting,
for any suitable function $f$,
\begin{equation}\label{4.4}
T_{\az,b}f:=b{T_\az}f-T_\az(bf).
\end{equation}

Now we state the main result of this section as follows.

\begin{theorem}\label{t4.2}
Let $\az\in(0,1)$ and $b\in\rm \rbmo$.
Assume that the generalized fractional integral
$T_\az$, defined by \eqref{4.3} associated with
kernel  $K_\az$ satisfying $\eqref{4.1}$ and
$\eqref{4.2}$, is bounded from $L^p(\mu)$ into
$L^q(\mu)$ for all $p\in(1,1/\az)$ and $1/q=1/p-\az$.
Then the commutator $T_{\az,b}$ defined by \eqref{4.4}
is bounded from $\wz H^1(\mu)$ into $L^{1/{(1-\az)},\fz}(\mu)$,
that is, there exists a positive constant $C$ such that,
for all $t\in(0,\fz)$ and all functions $f\in \wz H^1(\mu)$,
\begin{eqnarray*}
\lf[\mu(\{x\in\cx:\,|T_{\az,b}f(x)|>t\})\r]^{1-\az}
\leq C\|b\|_{\rbmo}t^{-1}\|f\|_{\wz H^1(\mu)}.
\end{eqnarray*}
\end{theorem}

To prove Theorem \ref{t4.2}, we need the following result
from \cite[Theorem 1.13]{fyy14t}.

\begin{lemma}\label{l4.3}
Let $\az\in(0,1)$ and $T_\az$ be as in \eqref{4.3} with
kernel $K_\az$ satisfying \eqref{4.1} and \eqref{4.2}.
Then the following statements are equivalent:
\begin{itemize}

\item[{\rm (i)}] $T_\az$ is bounded from $L^p(\mu)$ into
$L^q(\mu)$ for all $p\in(1,1/\az)$ and $1/q=1/p-\az$;

\item[{\rm (ii)}] $T_\az$ is bounded from $L^1(\mu)$ into
$L^{1/(1-\az),\fz}(\mu)$.

\end{itemize}

\end{lemma}

\begin{proof}[Proof of Theorem \ref{t4.2}]
For each fixed $f\in \wz H^1(\mu)$, by Definition \ref{d2.11},
we have a decomposition $f=\sum_{j=1}^\fz h_j$, where, for
any $j\in \nn$, $h_j$ is an $(\infty, 2-\az,12)_\lz$-atomic
block, $\supp h_j\subset S_j$, and
\begin{equation*}
\sum_{j=1}^\fz |h_j|_{\wz H_{\rm{atb},12}^
{1,\infty,2-\az}(\mu)}
\leq 2\|f\|_{\wz H^1 (\mu)}.
\end{equation*}
Moreover, for each fixed $j$, we can further decompose $h_j$
as $h_j=r_{j,1}a_{j,1}+r_{j,2}a_{j,2},$
where, for any $i\in\{1,2\}$, $r_{j,i}\in\cc$,
$a_{j,i}$ is a bounded function supported on some ball
$B_{j,i}\subset S_j$ satisfying
\begin{eqnarray*}
\|a_{j,i}\|_{L^\infty(\mu)}
\le\lf[\mu\lf(12B_{j,i}\r)
\lf\{\wz K_{B_{j,i},S_j}^{(12)}\r\}^{2-\az}\r]^{-1}
\end{eqnarray*}
and $ |h_j|_{\wz H_{\rm{atb},12}^{1,\infty,2-\az}(\mu)}
=|r_{j,1}|+|r_{j,2}|$. By Lemma \ref{l2.10}, we further
conclude that, for any $j$ and any $i\in\{1,2\}$,
\begin{eqnarray}\label{4.5}
\|a_{j,i}\|_{L^\infty(\mu)}
\ls\lf[\mu\lf(12B_{j,i}\r)
\lf\{\wz K_{B_{j,i},S_j}^{(6)}\r\}^{2-\az}\r]^{-1}.
\end{eqnarray}
Write
$$
T_{\az,b}f=\sum_{j=1}^\fz
\lf[b-m_{\wz{S_j}}(b)\r]T_\az h_j
+T_\az\lf(\sum_{j=1}^\fz \lf[m_{\wz{S_j}}(b)-
b(\cdot)\r]h_j\r)=:T_{\az,b}^{\rm I}f+T_{\az,b}^
{\rm{II}}f.
$$

By Lemma \ref{l4.3} and an argument completely analogous to that
used in the estimate for \eqref{3.8},
we conclude that
\begin{eqnarray}\label{4.6}
\lf[\mu\lf(\{x\in\cx:\,|T_{\az,b}^{\rm{II}}f(x)|>t\}\r)\r]^
{1-\az}\ls\|b\|_{\rbmo}t^{-1}\|f\|_{\wz H^1(\mu)}.
\end{eqnarray}

We now estimate $T_{\az,b}^{\rm I}f(x)$.
By the Minkowski integral inequality and the Fatou lemma,
we see that
\begin{eqnarray*}
&&\lf[\mu(\lf\{x\in{\cx}:\,|T_{\az,b}^{\rm I} f(x)|>t\r\})\r]^
{1-\az}\\
&&\hs\leq\lf\{\int_{\lf\{x\in\cx:\,|T_{\az,b}^{\rm I} f(x)|>t\r\}}
\frac{|\sum_{j=1}^\fz
[b(x)-m_{\wz S_j}(b)]T_\az h_j(x)|^{\frac1 {1-\az}}}
{t^{\frac1 {1-\az}}}\,d\mu(x)\r\}^{1-\az}\\
&&\hs\ls t^{-1}\sum_{j=1}^\fz
\lf[\int_{\lf\{x\in\cx:\,\lf|T_{\az,b}^{\rm I}f(x)\r|>t\r\}}
\lf|b(x)-m_{\wz S_j}(b)\r|^{\frac1 {1-\az}}
\lf|T_\az h_j(x)\r|^{\frac1 {1-\az}}\,d\mu(x)\r]^{1-\az}\\
&&\hs\ls t^{-1}\sum_{j=1}^\fz
\lf[\int_{6S_j} \lf|b(x)-m_{\wz S_j}(b)\r|^{\frac1 {1-\az}}
\lf|T_\az h_j(x)\r|^{\frac1 {1-\az}}\,d\mu(x)\r]^{1-\az}
+t^{-1}\sum_{j=1}^\fz \lf[\int_{\cx\backslash{6S_j}}
\cdots\r]^{1-\az}\\
&&\hs=:\rm G+\rm H.
\end{eqnarray*}
To estimate $\rm G$, for each fixed $j$, write
\begin{eqnarray*}
&&\int_{6S_j}\lf|b(x)-m_{\wz{S_j}}(b)\r|^{\frac1 {1-\az}}
|T_\az h_j(x)|^{\frac1 {1-\az}}\,d\mu(x)\\
&&\hs\ls|r_{j,1}|^{\frac1 {1-\az}}
\int_{6S_j} \lf|b(x)-m_{\wz{S_j}}(b)\r|^{\frac1 {1-\az}}
\lf|T_\az a_{j,1}(x)\r|^{\frac1 {1-\az}}\,d\mu(x)\\
&&\hs\hs +|r_{j,2}|^{\frac1 {1-\az}}
\int_{6S_j} \lf|b(x)-m_{\wz{S_j}}(b)\r|^{\frac1 {1-\az}}
\lf|T_\az a_{j,2}(x)\r|^{\frac1 {1-\az}}\,d\mu(x)
=:{\rm L}_{j,1}+{\rm L}_{j,2}.
\end{eqnarray*}
We only consider the term ${\rm L}_{j,1}$, the other term
${\rm L}_{j,2}$ can be estimated in a similar way, the
details being omitted. Write
\begin{eqnarray*}
{\rm L}_{j,1}&&\ls|r_{j,1}|^{\frac1 {1-\az}}
\int_{6S_j\backslash{6B_{j,1}}}
{\lf|b(x)-m_{\wz{S_j}}(b)\r|}^{\frac1 {1-\az}}
{\lf|T_\az a_{j,1}(x)\r|}^{\frac1 {1-\az}}\,d\mu(x)\\
&&\quad+|r_{j,1}|^{\frac1 {1-\az}}
\int_{6B_{j,1}}{\lf|b(x)-m_{\wz{6B_{j,1}}}(b)\r|}^
{\frac1 {1-\az}}
{\lf|T_\az a_{j,1}(x)\r|}^{\frac1 {1-\az}}\,d\mu(x)\\
&&\quad+|r_{j,1}|^{\frac1 {1-\az}}
{\lf|m_{\wz{6B_{j,1}}}(b)-m_{\wz{S_j}}(b)\r|}^{\frac1 {1-\az}}
\int_{6B_{j,1}}{\lf|T_\az a_{j,1}(x)\r|}^{\frac1 {1-\az}}\,d\mu(x)
=:{\rm U}_j+{\rm V}_j+{\rm W}_j.
\end{eqnarray*}
Let $p\in(1,1/\az)$ and $1/q=1/p-\az$. Let $\bz:=q(1-\az)$.
Then $\bz\in(1,\fz)$. Recall that $\frac1\bz+\frac1{\bz'}=1$.
It then follows, from the H\"older inequality, the assumption
that $T_\az$ is bounded from $L^p(\mu)$ into $L^q(\mu)$,
\eqref{2.8}, Lemma \ref{l2.9} and \eqref{4.5}, that
\begin{eqnarray*}
{\rm W}_j&&\ls\|b\|_{\rbmo}^{\frac1{1-\az}}|r_{j,1}|^
{\frac1{1-\az}}
\lf[\wz K_{\wz{6B_{j,1}},\wz{S_j}}^{(6)}\r]^{\frac1 {1-\az}}
\lf[\int_{6B_{j,1}}{\lf|T_\az a_{j,1}(x)\r|}^{\frac1
{1-\az}\bz}\,d\mu(x)\r]^{\frac1 \bz}
\lf[\mu(6B_{j,1})\r]^{\frac1 {{\bz}'}}\\
&&\ls\|b\|_{\rbmo}^{\frac1{1-\az}}|r_{j,1}|^{\frac1{1-\az}}
\lf[\wz K_{B_{j,1},S_j}^{(6)}\r]^{\frac1 {1-\az}}
\|T_\az a_{j,1}\|_{L^q(\mu)}^{\frac q \bz}
\lf[\mu(6B_{j,1})\r]^{\frac1 {{\bz}'}}\\
&&\ls\|b\|_{\rbmo}^{\frac1{1-\az}}|r_{j,1}|^{\frac1{1-\az}}
\lf[\wz K_{B_{j,1},S_j}^{(6)}\r]^{\frac1 {1-\az}}
\|a_{j,1}\|_{L^p(\mu)}^{\frac1{1-\az}}
\lf[\mu(6B_{j,1})\r]^{\frac1 {{\bz}'}}\\
&&\ls\|b\|_{\rbmo}^{\frac1 {1-\az}}|r_{j,1}|^{\frac1 {1-\az}}
\lf[\wz K_{B_{j,1},S_j}^{(6)}\r]^{\frac1 {1-\az}}
\lf[\mu(6B_{j,1})\r]^{\frac1 {1-\az}}
\|a_{j,1}\|_{L^\fz(\mu)}^{\frac1 {1-\az}}
\ls\|b\|_{\rbmo}^{\frac1 {1-\az}}|r_{j,1}|^{\frac1 {1-\az}},
\end{eqnarray*}
where, in the penultimate inequality, we used the fact that
\begin{equation}\label{4.7}
\frac1{p(1-\az)}+\frac1{\bz'}=\frac1{1-\az}.
\end{equation}
On the other hand, the H\"older inequality, together with
Corollary \ref{c2.17}, the boundedness from $L^p(\mu)$ into $L^q(\mu)$ of $T_\az$, \eqref{4.7} and \eqref{4.5},
implies that
\begin{eqnarray*}
{\rm V}_j&&\ls\lf|r_{j,1}\r|^{\frac1 {1-\az}}
\lf[\int_{6B_{j,1}}
\lf|b(x)-m_{\wz{6B_{j,1}}}(b)\r|^{\frac1
{1-\az}\bz'}\,d\mu(x)\r]^{\frac1 {{\bz}'}}
\|T_\az a_{j,1}\|_{L^q(\mu)}^{\frac q \bz}\\
&&\ls\lf|r_{j,1}\r|^{\frac1 {1-\az}}
\|b\|_{\rbmo}^{\frac1 {1-\az}}
\lf[\mu(12B_{j,1})\r]^{\frac1 {\bz'}}
\|T_\az a_{j,1}\|_{L^q(\mu)}^{\frac1{1-\az}}\\
&&\ls\lf|r_{j,1}\r|^{\frac1 {1-\az}}
\|b\|_{\rbmo}^{\frac1 {1-\az}}
\|a_{j,1}\|_{L^p(\mu)}^{\frac1{1-\az}}
\lf[\mu(12B_{j,1})\r]^{\frac1 {\bz'}}\\
&&\ls\|b\|_{\rbmo}^{\frac1 {1-\az}}
|r_{j,1}|^{\frac1 {1-\az}}
\lf[\mu(12B_{j,1})\r]^{\frac1 {1-\az}}
\|a_{j,1}\|_{L^\fz(\mu)}^{\frac1 {1-\az}}
\ls\|b\|_{\rbmo}^{\frac1 {1-\az}}|r_{j,1}|^{\frac1 {1-\az}}.
\end{eqnarray*}
To estimate ${\rm U}_j$, we first observe that, for any
$x\notin 6B_{j,1}$ and $y\in B_{j,1}$,
$d(x,y)\sim d(x,c_{B_{j,1}})$. It then follows
from \eqref{4.1} that, for any $x\notin 6B_{j,1}$,
\begin{eqnarray*}
|T_\az a_{j,1}(x)|&&=\lf|\int_{\cx}
K_\az(x,y)a_{j,1}(y)\,d\mu(y)\r|
\le\int_{B_{j,1}} |K_\az(x,y)||a_{j,1}(y)|\,d\mu(y)\\
&&\ls{\frac{\|a_j\|_{L^1(\mu)}}
{[\lz(x,d(x,c_{B_{j,1}}))]^{1-\az}}}
\ls{\frac{\|a_j\|_{L^\fz(\mu)}\mu(B_{j,1})}
 {[\lz(x,d(x,c_{B_{j,1}}))]^{1-\az}}}.
\end{eqnarray*}
Let $N_1:=N_{6B_{j,1},6S_j}^{(6)}
+\lfloor \log_{6}2\rfloor+1$.
A straightforward computation, via the above estimate,
Corollary \ref{c2.17}, \eqref{2.8}, Lemma \ref{l2.9},
\eqref{2.2} and \eqref{4.5}, shows that
\begin{eqnarray*}
{\rm U}_j&&\ls|r_{j,1}|^{\frac1 {1-\az}}
\|a_{j,1}\|_{L^\fz(\mu)}^{\frac1 {1-\az}}
\lf[\mu(B_{j,1})\r]^{\frac1 {1-\az}}
\int_{6S_j\backslash{6B_{j,1}}}
\frac{|b(x)-m_{\wz{S_j}}(b)|^{\frac1 {1-\az}}}
{\lz(x,d(x,c_{B_{j,1}}))}\,d\mu(x)\\
&&\ls|r_{j,1}|^{\frac1 {1-\az}}
\|a_{j,1}\|_{L^\fz(\mu)}^{\frac1 {1-\az}}
\lf[\mu(B_{j,1})\r]^{\frac1 {1-\az}}
\sum_{k=1}^{N_1}
\lf[\int_{6^{k+1}B_{j,1}\backslash{6^kB_{j,1}}}
\frac{|b(x)-m_{\wz{6^{k+1}B_{j,1}}}(b)|^{\frac1 {1-\az}}}
{\lz(x,d(x,c_{B_{j,1}}))}\,d\mu(x)\r.\\
&&\lf.\quad+\lf|m_{\wz{6^{k+1}B_{j,1}}}(b)
-m_{\wz{S_j}}(b)\r|^{\frac1 {1-\az}}
\int_{6^{k+1}B_{j,1}\backslash{6^kB_{j,1}}}
\frac1 {\lz(x,d(x,c_{B_{j,1}}))}\,d\mu(x)\r]\\
&&\ls\|b\|_{\rbmo}^{\frac1 {1-\az}}|r_{j,1}|^{\frac1 {1-\az}}
\|a_{j,1}\|_{L^\fz(\mu)}^{\frac1 {1-\az}}
\lf[\mu(B_{j,1})\r]^{\frac1 {1-\az}}\\
&&\hs\times\sum_{k=1}^{N_1} \lf[\frac{\mu(6^{k+2}B_{j,1})}
{\lz(c_{B_{j,1}},6^{k+2}r_{B_{j,1}})}
+\lf\{\wz K_{B_{j,1},S_j}^{(6)}\r\}^{\frac1 {1-\az}}
\frac{\mu(6^{k+1}B_{j,1})}
{\lz(c_{B_{j,1}},6^{k+1}r_{B_{j,1}})}\r]\\
&&\ls\|b\|_{\rbmo}^{\frac1 {1-\az}}|r_{j,1}|^{\frac1 {1-\az}}
\|a_{j,1}\|_{L^\fz(\mu)}^{\frac1 {1-\az}}
\lf[\mu(B_{j,1})\r]^{\frac1 {1-\az}}
\lf[\wz K_{B_{j,1},S_j}^{(6)}\r]^{\frac1 {1-\az}}\\
&&\hs\times\lf[1+\sum_{k=1}^{N_1}
\frac{\mu(6^{k}B_{j,1})} {\lz(c_{B_{j,1}},6^{k}r_{B_{j,1}})}\r]\\
&&\ls\|b\|_{\rbmo}^{\frac1 {1-\az}}|r_{j,1}|^{\frac1 {1-\az}}
\|a_{j,1}\|_{L^\fz(\mu)}^{\frac1 {1-\az}}
\lf[\mu(B_{j,1})\r]^{\frac1 {1-\az}}
\lf[\wz K_{B_{j,1},S_j}^{(6)}\r]^{\frac{2-\az}{1-\az}}
\ls\|b\|_{\rbmo}^{\frac1 {1-\az}}|r_{j,1}|^{\frac1 {1-\az}}.
\end{eqnarray*}
Combining the estimates for ${\rm U}_j$, ${\rm V}_j$ and
${\rm W}_j$, we obtain $${\rm L}_{j,1}\ls\|b\|_{\rbmo}^
{\frac1 {1-\az}}|r_{j,1}|^{\frac1 {1-\az}},$$
which further implies that
$${\rm G}\ls t^{-1}\|b\|_{\rbmo}\sum_{j=1}^\fz
|h_j|_{\wz H_{\rm{atb},12}^{1,\infty,2-\az}(\mu)}
\ls t^{-1}\|b\|_{\rbmo}\|f\|_{\wz H^1(\mu)}.$$

It remains to estimate $\rm H$. Observe that, for any
$x\notin 6S_j$ and $y\in S_j$, $d(x,y)\sim d(x,c_{S_j})$.
From this, together with the vanishing moment of $h_j$
and \eqref{4.2}, we deduce that, for any $x\notin 6S_j$,
\begin{eqnarray*}
|T_\az h_j(x)|&&=\lf|\int_{\cx} K_\az(x,y)h_j(y)\,d\mu(y)\r|
\leq\int_{S_j} |K_\az(x,y)-K_\az(x,c_{S_j})||h_j(y)|\,d\mu(y)\\
&&\ls\int_{S_j} {\frac{[d(y,c_{S_j})]^\delta}
{[d(x, y)]^\delta[\lz(x,d(x, y))]^{1-\az}}}|h_j(y)|\,d\mu(y)\\
&&\ls{\frac{(r_{S_j})^\delta}  {[d(x,c_{S_j})]^\delta
[\lz(x,d(c_{S_j},x))]^{1-\az}}}\|h_j\|_{L^1(\mu)}.
\end{eqnarray*}
On the other hand, a trivial computation, via Corollary
\ref{c2.17}, \eqref{3.9} and \eqref{2.1}, gives us that,
for any $i\in\nn$,
\begin{eqnarray*}
&&{\frac1 {\lz(c_{S_j},6^ir_{S_j})}} \int_{6^{i+1}S_j}
\lf|b(x)-m_{\wz{S_j}}(b)\r|^{\frac1 {1-\az}}\,d\mu(x)\\
&&\hs\leq {\frac1 {\lz(c_{S_j},6^ir_{S_j})}}\lf[\int_{6^{i+1}S_j}  \lf|b(x)-m_{\wz{6^{i+1}S_j}}(b)\r|^{\frac1 {1-\az}}\,d\mu(x)\r.\\
&&\hs\hs\lf.+\mu(6^{i+1}S_j)
\lf|m_{\wz{6^{i+1}S_j}}(b)-m_{\wz{S_j}}(b)\r|^{\frac1 {1-\az}}\r]\\
&&\hs\ls\frac{\mu(2\cdot6^{i+1}S_j)} {\lz(c_{S_j},6^ir_{S_j})}
\|b\|_{\rbmo}^{\frac1 {1-\az}}
+\frac{\mu(6^{i+1}S_j)} {\lz(c_{S_j},6^ir_{S_j})}
\lf[i\|b\|_{\rbmo}\r]^{\frac1 {1-\az}}
\ls\lf[i\|b\|_{\rbmo}\r]^{\frac1 {1-\az}}.
\end{eqnarray*}
By the above two estimates, we conclude that
\begin{eqnarray*}
{\rm H}&&=t^{-1}\sum_{j=1}^\fz
\lf[\int_{\cx\backslash{6S_j}}
\lf|b(x)-m_{\wz{S_j}}(b)\r|^{\frac1 {1-\az}}
\lf|\int_{S_j} K_\az(x,y)h_j(y)\,d\mu(y)\r|^
{\frac1 {1-\az}}\,d\mu(x)\r]^{1-\az}\\
&&\ls t^{-1}\sum_{j=1}^\fz \lf\{\sum_{i=1}^\fz
\int_{6^{i+1}S_j\backslash{6^iS_j}}
|b(x)-m_{\wz{S_j}}(b)|^{\frac1 {1-\az}}\r.\\
&&\hs\lf.\times\lf[{\frac{(r_{S_j})^\delta}
{(d(x,c_{S_j}))^\delta(\lz(x,d(c_{S_j},x)))^{1-\az}}}
\|h_j\|_{L^1(\mu)}\r]^{\frac1 {1-\az}}\,d\mu(x)\r\}^{1-\az}\\
&&\ls t^{-1}\sum_{j=1}^\fz \|h_j\|_{L^1(\mu)}
\lf[\sum_{i=1}^\fz ({6^i})^{{\frac{-\delta} {1-\az}}}
{\frac1 {\lz(c_{S_j},6^ir_{S_j})}}
\int_{6^{i+1}S_j} |b(x)-m_{\wz{S_j}}(b)|^
{\frac1 {1-\az}}\,d\mu(x) \r]^{1-\az}\\
&&\ls t^{-1}\|b\|_{\rbmo}\sum_{j=1}^\fz \|h_j\|_{L^1(\mu)}
\lf[\sum_{i=1}^\fz \lf(6^{\frac{-i\delta} {1-\az}}\r)
(i)^{\frac1 {1-\az}} \r]^{1-\az}
\ls t^{-1}\|b\|_{\rbmo}\|f\|_{\wz H^1(\mu)}.
\end{eqnarray*}

We finally obtain
\begin{eqnarray*}
\mu\lf(\{x\in\cx:\,|T_b^{\rm I}f(x)|>t\}\r)^{1-\az}
\ls t^{-1}\|b\|_{\rbmo}\|f\|_{\wz H^1(\mu)},
\end{eqnarray*}
which, together with the estimate $\eqref{4.6}$,
completes the proof of Theorem \ref{t4.2}.
\end{proof}
\section{Boundedness of Commutators $T_b$ on $L^p(\mu)$
with $p\in(1,\fz)$}\label{s5}
\hskip\parindent
In this section, we establish the boundedness on $L^p(\mu)$,
for all $p\in(1,\fz)$, of the commutator $T_b$
generated by $b\in\rbmo$ and the Caler\'{o}n-Zygmund operator
$T$ with kernel satisfying \eqref{3.1}
and \eqref{3.2}, which improves \cite[Theorem 3.10]{fyy12}.

\begin{theorem}\label{t5.1}
Let $b\in\rm \rbmo$. Assume that the Calder\'on-Zygmund
operator $T$, defined by \eqref{3.3} associated with
kernel $K$ satisfying $\eqref{3.1}$ and $\eqref{3.2}$,
is bounded on $L^2(\mu)$. Then the commutator $T_b$ defined
in \eqref{3.4} is bounded on $L^p(\mu)$ with $p\in{(1,\fz)}$.
\end{theorem}

To prove Theorem \ref{t5.1}, we borrow some ideas from the
proof of \cite[Theorem 4.3]{hmy05}. We need several tools,
including an  interpolation theorem for sublinear operators,
which is an extension of \cite[Theorem 3.1]{hmy05} and whose
proof is different from that of \cite[Theorem 3.1]{hmy05}.
We start with the notion of maximal functions in
\cite{ad13, h10}.

For any $f\in L^1_{loc}(\mu)$, the \emph{sharp maximal
function $M^\sharp f$} is defined by setting, for all
$x\in\cx$,
$$M^\sharp f(x):=\sup_{B\ni x}\frac1 {\mu(6B)}
\int_B |f(y)-m_{\wz{B}}(f)|\,d\mu(y)+
\sup_{\begin{subarray}{c}x\in B\subset S\\
B,\,S\ (6,\bz_6)-{\rm doubling}\end{subarray}}
\frac{|m_B(f)-m_S(f)|} {\wz K_{B,S}^{(6)}},$$
where the first supremum is taking over all balls $B$
containing $x$. By an argument similar to that used in
the proof of \cite[Lemma 3.1]{hmy05}, we have the
following result.
\begin{lemma}\label{l5.2}
For $r\in(0,1)$ and any $|f|^r\in L^1_{loc}(\mu)$, let
$M^\sharp_r f:=\lf[M^\sharp\lf(|f|^r\r)\r]^{\frac1 r}$.
Then there exists a positive constant $C_{(r)}$,
depending on $r$, such that, for all
$|f|^r\in L^1_{loc}(\mu)$,
$M^\sharp_r f\leq C_{(r)}M^\sharp f$.
\end{lemma}

The \emph{non-centered doubling maximal operator} $N$
is defined by setting, for all $f\in L^1_\loc(\mu)$
and $x\in\cx$,
$$ Nf(x):=\sup_{\begin{subarray}{c}
B\ni x\\
B\,(6,\bz_6)-{\rm doubling}\end{subarray}}
\frac1 {\mu(B)}\int_B |f(y)|\,d\mu(y).$$
By the Lebesgue differential theorem, it is easy to
see that, for any $f\in L^1_{loc}(\mu)$ and
$\mu$-almost every $x\in\cx$,
\begin{eqnarray}\label{5.1}
|f(x)|\leq Nf(x)
\end{eqnarray}
(see \cite[Corollary 3.6]{h10}). The following lemma
is just \cite[Lemma 3.3]{lyy12}.

\begin{lemma}\label{l5.3}
For all $f\in L_{loc}^1(\mu)$, with
\begin{equation*}
\int_\cx f(y)\,d\mu(y)=0\quad{\rm when}\,\,\mu(\cx)<\fz,
\end{equation*}
if $\min\{1,Nf\}\in L^{p_0}(\mu)$ for some $p_0\in(1,\fz)$,
then, for all $p\in[p_0,\fz)$, there exists a positive
constant $C_{(p)}$, depending on $p$ but independent of
$f$, such that
\begin{equation*}
\lf\|Nf\r\|_{L^{p,\fz}(\mu)}
\leq C_{(p)}\lf\|M^\sharp f\r\|_{L^{p,\fz}(\mu)}.
\end{equation*}
\end{lemma}

Moreover, for $r\in(0,\fz)$ and $\eta\in(1,\fz)$, the
\emph{maximal operator} $M_{r,(\eta)}$ is defined by
setting, for all $f\in L^r_\loc(\mu)$ and $x\in\cx$,
\begin{equation*}
M_{r,(\eta)}f(x):=\sup_{B\ni x}
\lf[\frac1 {\mu(\eta B)}\int_B |f(y)|^r \,d\mu(y)\r]
^{\frac1 r},
\end{equation*}
where the supremum is taking over all balls $B$
containing $x$. With a proof similar to that of
\cite[Lemma 2.3]{hlyy}, we obtain the following useful
properties of $M_{r,(\eta)}$, the details being omitted.
\begin{lemma}\label{l5.4}
The following statements hold true:

{\rm (i)} Let $p\in(1,\fz)$, $r\in(1,p)$ and $\eta\in[5,\fz)$.
Then $M_{r,(\eta)}$ is bounded on $L^p(\mu)$.

{\rm (ii)} Let $r\in(0,1)$ and $\eta\in[5,\fz)$.
Then $M_{r,(\eta)}$ is bounded on $L^{1,\fz}(\mu)$, that is,
there exists a positive constant $C$ such that, for all
$f\in L^{1,\fz}(\mu)$,
\begin{eqnarray}\label{5.2}
\sup_{\sigma>0} \sigma\mu\lf(\{x\in\cx:\,
M_{r,(\eta)}f(x)>\sigma  \}\r)
\leq C\sup_{\sigma>0}\sigma\mu\lf(\{x\in\cx:\,
|f(x)|>\sigma  \}\r).
\end{eqnarray}
\end{lemma}

Now we state our interpolation theorem for sublinear operators
as follows.
\begin{theorem}\label{t5.5}
Let $p_0\in (1,\fz)$, $k\in\nn$ and $T_i$ be a sublinear operator
for all $i\in\{1,\ldots,k\}$. Suppose that

{\rm (i)} $T_1$ is bounded from $\wz H^1(\mu)$ into $L^{1,\fz}(\mu)$, that is, there exists a positive constant $C$
such that, for all $f\in \wz H^1(\mu)$ and $t\in(0,\fz)$,
\begin{equation*}
\mu\lf(\{x\in\cx:\,|T_1f(x)|>t\}\r)\leq \frac C t \|f\|_{\wz H^1(\mu)};
\end{equation*}

{\rm (ii)} when $\mu(\cx)<\fz$, for any $p\in(1,\fz)$, there exists
a positive constant $C_{(p)}$, depending on $p$, such that,
for all $f\in L^\fz_b(\mu)$,
$$\frac1{\mu(\cx)}\int_\cx|T_1f(x)|\,d\mu(x)\le C\|f\|_{L^p(\mu)};$$

{\rm (iii)} there exists a positive constant $D$ such that,
for all $f\in L^{\fz}_b(\mu)$,
\begin{equation*}
M^\sharp (T_1f)\leq \sum_{i=2}^k|T_if|+D\|f\|_{L^{\fz}(\mu)};
\end{equation*}

{\rm (iv)} $T_i$ is bounded on $L^{p_0}(\mu)$ for
all $i\in\{2,\ldots,k\}$.

Then $T_1$ is bounded on $L^p(\mu)$ for all $p\in(1,p_0)$.
\end{theorem}
\begin{proof}
Let $r\in(0,1)$. We first claim that, for all $p\in(1,p_0)$,
there exists a positive constant $C$ such that, for all
$f\in L^\fz_b(\mu)$,
\begin{eqnarray}\label{5.3}
\sup_{t>0}t^p\mu\lf(\{x\in\cx:\,|M_r^\sharp(T_1f)(x)|>t\}\r)
\leq C\|f\|_{L^p(\mu)}^p.
\end{eqnarray}
To show \eqref{5.3}, notice that, although it is unclear
whether the operator $M_r^\sharp(T_1)$ is quasi-linear or not,
we still conclude that there exists a positive constant $\wz C_{(r)}$,
depending on $r$, such that, for all
$f_1,f_2\in L_b^{\fz}(\mu)$,
\begin{eqnarray}\label{5.4}
M_r^\sharp\lf(T_1(f_1+f_2)\r)
\leq \wz C_{(r)}\lf[M_{r,(6)}(T_1f_1)+M_r^\sharp(T_1f_2)\r].
\end{eqnarray}
Indeed, by $r\in(0,1)$, we see that, for any
$f_1,f_2\in L_b^{\fz}(\mu)$,
$x\in \cx$ and ball $B\ni x$,
\begin{eqnarray}\label{5.5}
&&\frac1 {\mu(6B)} \int_B |
|T_1(f_1+f_2)(y)|^r-m_{\wz{B}}(|T_1(f_1+f_2)|^r)  |\,d\mu(y)\\
&&\quad\leq \frac1 {\mu(6B)} \int_B
\lf||T_1f_2(y)|^r-m_{\wz{B}}(|T_1f_2|^r)\r|\,d\mu(y)
+m_{\wz{B}}(|T_1f_1|^r)\noz\\
&&\quad\quad +\frac1 {\mu(6B)} \int_B |T_1f_1(y)|^r\,d\mu(y)
\ls\lf[M_r^\sharp(T_1f_2)(x) \r]^r
+\lf[M_{r,(5)}(T_1f_1)(x)\r]^r\noz
\end{eqnarray}
and, for any $(6,\bz_6)$-doubling balls $B\subset S$ with $B\ni x$,
\begin{eqnarray}\label{5.6}
&&|m_B(|T_1(f_1+f_2)|^r)-m_S(|T_1(f_1+f_2)|^r)  |\\
&&\quad\leq |m_B(|T_1f_2|^r)-m_S(|T_1f_2|^r)|
+m_B(|T_1f_1|^r)+m_S\lf(|T_1f_1|^r\r)\noz\\
&&\quad\ls\wz K_{B,S}^{(6)}
\lf[M_r^\sharp(T_1f_2)(x)\r]^r+\lf[M_{r,(6)}(T_1f_1)(x)\r]^r\noz.
\end{eqnarray}
Combining \eqref{5.5} and \eqref{5.6}, we obtain \eqref{5.4}.

For each fixed $t\in(0,\fz)$ and each $f\in L^\fz_b(\mu)$,
applying the Calder\'{o}n-Zygmund
decomposition to $|f|^p$ at level $t^p$
(with $t^p> {\gz}_0\|f\|_{L^p(\mu)}/{\mu(\cx)}$ if $\mu(\cx)< \fz$)
with the notation as in Lemma \ref{l2.18}, we obtain $f=g+h$,
where $g:=f\chi_{\cx\backslash(\cup_j{6B_j})}+\Sigma_j\varphi_j$ and
$h:=f-g=\sum_j[\omega_jf-\varphi_j]$. Moreover, we have
\begin{equation}\label{5.7}
\|g\|_{L^\fz(\mu)}\le(\gz+1)t\quad{\rm and}\quad\|g\|_{L^p(\mu)}
\ls\|f\|_{L^p(\mu)}.
\end{equation}
From this, together with \eqref{5.4}, Lemma $\ref{l5.2}$ and
Theorem \ref{t5.5}(iii), we deduce that
\begin{eqnarray*}
&&\mu\lf(\lf\{x\in\cx:\,M_r^\sharp(T_1f)(x)>
3D\wz C_{(r)}C_{(r)}(\gz+1) t\r\}\r)\\
&&\hs\leq\mu\lf(\lf\{x\in\cx:\,M_r^\sharp(T_1g)(x)>
2DC_{(r)}(\gz+1) t  \r\}\r)\\
&&\hs\hs+\mu\lf(\lf\{x\in\cx:\,M_{r,(6)}(T_1h)(x)>
DC_{(r)}(\gz+1) t  \r\}\r)\\
&&\hs\leq\mu\lf(\lf\{x\in\cx:\,M^\sharp(T_1g)(x)>
2D(\gz+1)t \r \}\r)\\
&&\hs\hs+\mu\lf(\lf\{x\in\cx:\,M_{r,(6)}(T_1h)(x)>
DC_{(r)}(\gz+1)t  \r\}\r)\\
&&\hs\le\sum_{i=2}^k\mu\lf(\lf\{x\in\cx:|T_ig(x)|>
D(\gz+1)t/(k-1) \r \}\r)\\
&&\hs\hs+\mu\lf(\lf\{x\in\cx:\,M_{r,(6)}(T_1h)(x)>
DC_{(r)}(\gz+1)t  \r\}\r).
\end{eqnarray*}
For $i\in\{2,\ldots,k\}$, the boundedness of $T_i$ on
$L^{p_0}(\mu)$ and \eqref{5.7} imply that
\begin{equation}\label{5.8}
\mu(\{x\in\cx:\,|T_ig(x)|>D(\gz+1)t \})
\ls t^{-p_0}\|T_ig\|_{L^{p_0}(\mu)}^{p_0}
\ls t^{-p}\|f\|^p_{L^p(\mu)}.
\end{equation}
By Lemma \ref{l5.4}(ii), Theorem \ref{t5.5}(i) and Lemma
\ref{l2.18}(iii), we have
\begin{eqnarray*}
\mu\lf(\lf\{x\in\cx:M_{r,(6)}(T_1h)(x)>DC_{(r)}(\gz+1)t\r\}\r)
&&\ls t^{-1}\sup_{\sigma>0}\sigma
\mu(\{x\in\cx:\,|T_1h(x)|>\sigma  \})\\
&&\ls t^{-1}\|h\|_{\wz H^1(\mu)}
\ls t^{-p}\|f\|_{L^p(\mu)}^p,
\end{eqnarray*}
which, together with the estimate $\eqref{5.8}$,
implies \eqref{5.3}.

We now conclude the proof of Theorem \ref{t5.5} by considering
the following two cases.

{\it Case (i)} $\mu(\cx)=\fz$. Let
$$L^\fz_{b,\,0}(\mu):=
\lf\{f\in L_b^\fz(\mu):\,\int_\cx f(x)\,d\mu(x)=0\r\}.$$
Then, in this case, $L^\fz_{b,\,0}(\mu)$ is dense in $L^p(\mu)$
for all $p\in(1, \fz)$. Let $N_rf:= [N(|f|^r)]^{1/r}$ for
all $f\in L^r_{\rm loc}(\mu)$.
We now show that, for all $f\in L^\fz_{b,\,0}(\mu)$,
$\min\{1, N_r(T_1f)\}\in L^p(\mu)$ for all $p\in(1,\fz)$.
Indeed, for all $f\in L^\fz_{b,\,0}(\mu)$, we see that
$f\in {\wz H^1(\mu)}$. Moreover, by the definitions of $N_r$
and $M_{r,(6)}$ with $r\in(0,1)$, we know that,
for $\mu$-almost every $x\in\cx$,
$N_r(T_1f)(x)\ls M_{r,(6)}(T_1f)(x)$.
It then follows, from Lemma \ref{l5.4}(ii)
and Theorem \ref{t5.5}(i), that, for all $r\in(0,1)$,
$$
\|N_r(T_1f)\|_{L^{1,\fz}(\mu)}
\ls\|M_{r,(6)}(T_1f)\|_{L^{1,\fz}(\mu)}
\ls\|T_1f\|_{L^{1,\fz}(\mu)}
\ls\|f\|_{\wz H^1(\mu)},
$$
which implies that, for all $p\in(1,\fz)$,
\begin{eqnarray*}
&&\int_{\cx} \lf[\min\{1,N_r(T_1f)\}(x)\r]^p\,d\mu(x)\\
&&\quad =p\int_0^2 t^{p-1}
\mu(\{x\in\cx:\,\min\{1,N_r(T_1f) \}(x)>t  \})\,dt
+p\int_2^\fz \cdots\\
&&\quad\ls p\int_0^2 t^{p-1}
\mu(\{x\in\cx:\,\min\{1,N_r(T_1f) \}(x)>t  \})\,dt\\
&&\quad\ls\|N_r(T_1f)\|_{L^{1,\fz}(\mu)}\int_0^2 t^{p-2}\,dt
\ls\|N_r(T_1f)\|_{L^{1,\fz}(\mu)}<\fz.
\end{eqnarray*}
Thus, for all $f\in L^\fz_{b,\,0}(\mu)$,
$\min\{1,N_r(T_1f)\}\in L^p(\mu)$.
From this, \eqref{5.1}, Lemma \ref{l5.3} and \eqref{5.3},
we deduce that, for all $f\in L^\fz_{b,\,0}(\mu)$
and all $p\in(1,p_0)$,
\begin{eqnarray*}
\|T_1f\|_{L^{p,\fz}(\mu)}
&&\leq \lf\|N_r(T_1f)\r\|_{L^{p,\fz}(\mu)}
=\lf\|N(|T_1f|^r)\r\|_{L^{p/r,\fz}(\mu)}^{1/r}
\ls\lf\|M^\sharp(|T_1f|^r)\r\|_{L^{p/r,\fz}(\mu)}^{1/r}\\
&&\sim\lf\|M_r^{\sharp}(T_1f)\r\|_{L^{p,\fz}(\mu)}
\ls\|f\|_{L^p(\mu)},
\end{eqnarray*}
which, along with the Marcinkiewicz interpolation theorem
and a standard density argument, implies that, for all
$p\in(1,p_0)$ and $f\in L^p(\mu)$,
\begin{eqnarray*}
\|T_1f\|_{L^p(\mu)}\ls\|f\|_{L^p(\mu)}.
\end{eqnarray*}

{\it Case (ii)} $\mu(\cx)<\fz$.  In this case, for
all $r\in(0,1)$, $f\in L_b^{\fz}(\mu)$ and $x\in\cx$,
we see that
\begin{eqnarray*}
|T_1f(x)|&&\leq \lf[N(|T_1f|^r)(x)\r]^{1/r}\\
&&\ls\lf\{N\lf(|T_1f|^r-{\frac1 {\mu(\cx)}}
\int_{\cx}|T_1f(y)|^r\,d\mu(y)   \r)(x)\r\}^{1/r}\\
&&\quad+\lf\{\frac1 {\mu(\cx)}
\int_{\cx} |T_1f(x)|^r\,d\mu(x)  \r\}^{1/r}
=:{\rm E}(x)+{\rm F}^{1/r}.
\end{eqnarray*}
Observe that $\int_\cx [|T_1f(y)|^r-{\rm F}]\,d\mu(y)=0$ and,
for all $p\in(1,\fz)$,
$$\int_\cx \lf[\min\{1,N(|T_1f|^r-{\rm F})(x)\}\r]^p\,d\mu(x)
<\mu(\cx)<\fz.$$
From this, together with Lemma \ref{l5.3}, $M^\sharp({\rm F})=0$
and \eqref{5.3}, we deduce that, for all $p\in(1,p_0)$,
$$\lf\|N(|T_1f|^r-\rm F)\r\|_{L^{p/r,\fz}(\mu)}^{1/r}
\ls\lf\|M^\sharp(|T_1f|^r)\r\|_{L^{p/r,\fz}(\mu)}^{1/r}
\sim\lf\|M_r^{\sharp}(T_1f)\r\|_{L^{p,\fz}(\mu)}
\ls\|f\|_{L^p(\mu)}.$$
A trivial computation via the H\"older inequality and
Theorem \ref{t5.5}(ii) leads to, for all $p\in(1,\fz)$,
$$ {\rm F}^{1/r}\ls\frac1{\mu(\cx)}\int_\cx|T_1f(x)|\,d\mu(x)
\ls\|f\|_{L^p(\mu)}.$$
Combining the above two estimates, we obtain the desired
conclusion also in this case, which completes the proof of
Theorem \ref{t5.5}.
\end{proof}

To prove Theorem \ref{t5.1}, we also need the following
pointwise estimate.

\begin{theorem}\label{t5.6}
Let $b\in{L^\fz(\mu)}$, the operator $T$ with kernel $K$ be
the same as in Theorem \ref{t3.1} and $T_b$ as
in \eqref{3.4}. Suppose that $T$ is bounded on $L^2(\mu)$.
Then, for any $s\in{(1,\fz)}$, there exists a positive
constant $C_{(s)}$, depending on $s$, such that,
for all $f\in L_b^\fz(\mu)$,
\begin{equation}\label{5.9}
M^\sharp\lf(T_bf\r)\leq C_{(s)}\|b\|_{\rbmo}
\lf[\|f\|_{L^{\fz}(\mu)}
+M_{s,(5)}f+M_{s,(6)}Tf+T_*f\r],
\end{equation}
where $T_*$ denotes the \emph{maximal Calder\'on-Zygmund operator}
defined by setting, for all $f\in L^\fz_b(\mu)$ and $x\in\cx$,
$$ T_*f(x):=\sup_{\varepsilon>0}
\lf|\int_{d(x,y)>\varepsilon} K(x,y)f(y)\,d\mu(y)\r|.$$
\end{theorem}

To prove Theorem \ref{t5.6}, we begin with the following technical
lemma from \cite[Lemma 3.13]{fyy12}.

\begin{lemma}\label{l5.7}
There exists a positive constant $P_0$ (big enough),
depending on $C_{(\lz)}$ in \eqref{2.1} and $\bz_6$ as
in \eqref{2.3}, such that, if $x\in\cx$ is some fixed point
and $\{f_B\}_{B\ni x}$ is a collection of numbers such that
$|f_B-f_S|\leq C_{(x)}$ for all doubling balls $B\subset S$
with $x\in B$ such that $\wz K_{B,S}^{(6)}\leq P_0$,
then there exists a positive constant $C$,
depending only on $C_{(\lz)}$, $\bz_6$ and $P_0$,
such that, for all doubling balls $B\subset S$ with $x\in B$,
\begin{eqnarray*}
|f_B-f_S|\leq C\wz K_{B,S}^{(6)}C_{(x)}.
\end{eqnarray*}
\end{lemma}

\begin{proof}[Proof of Theorem \ref{t5.6}]
We first show that, for all $x$ and balls $B$ with $B\ni x$,
\begin{eqnarray}\label{5.10}
&&\frac1 {\mu(6B)}\int_B |T_bf(y)-h_B|\,d\mu(y)\\
&&\quad\ls\|b\|_{\rbmo}\lf[M_{s,(5)}f(x)+M_{s,(6)}Tf(x)
+\|f\|_{L^{\fz}(\mu)}\r]\noz
\end{eqnarray}
and, for all balls $B\subset S$ with $B\ni x$,
\begin{equation}\label{5.11}
|h_B-h_S|\ls\|b\|_{\rbmo}\lf[\wz K_{B,S}^{(6)}\r]^2
\lf[M_{s,(5)}f(x)+T^*f(x)+\|f\|_{L^{\fz}(\mu)} \r],
\end{equation}
where
$$ h_B:=-m_B\lf(T([b-m_{\wz{B}}(b)]
f\chi_{{\cx}\backslash{(6/5)B}})\r)$$
and
$$ h_S:=-m_S\lf(T([b-m_{\wz{S}}(b)]
f\chi_{{\cx}\backslash{(6/5)S}})\r).$$
The hypotheses $b\in{L^{\fz}(\mu)}$ and $f\in{L^\fz_b(\mu)}$
imply that $h_B$ and
$h_S$ are both finite.

The proof of \eqref{5.10} is analogous to that of
\cite[(10)]{hmy05} with a slight modification, the details
being omitted.

We now show \eqref{5.11}. For any two balls $B\subset S$
with $B\ni x$, let
$N_2:=N_{B,S}^{(6)}+\lfloor \log_{\rho}2\rfloor+2$.
Write
\begin{eqnarray*}
&&|h_B-h_S|\\
&&\quad=\lf|m_B\lf(T\lf(\lf[b-m_{\wz{B}}(b)\r]f
\chi_{{\cx}\backslash{(6/5)
B}}\r)\r)-m_S\lf(T\lf(\lf[b-m_{\wz{S}}(b)\r]f
\chi_{{\cx}\backslash{(6/5)S}}\r)\r)\r|\\
&&\quad\leq \lf|m_B\lf(T\lf(\lf[b-m_{\wz{B}}(b)\r]f
\chi_{{6B}\backslash{(6/5)B}}\r)\r)\r|
+\lf|m_B\lf(T\lf(\lf[m_{\wz{B}}(b)-m_{\wz{S}}(b)\r]f
\chi_{{\cx}\backslash{6B}}\r)\r)\r|\\
&&\quad\quad+\lf|m_B\lf(T\lf(\lf[b-m_{\wz{S}}(b)\r]f
\chi_{{6^{N_2}B}\backslash{6B}}\r)\r) \r |\\
&&\quad\quad+\lf|m_B\lf(T\lf(\lf[b-m_{\wz{S}}(b)\r]f
\chi_{{\cx}\backslash{6^{N_2}B}}\r)\r)
-m_S\lf(T\lf(\lf[b-m_{\wz{S}}(b)\r]f
\chi_{{\cx}\backslash{6^{N_2}B}}\r)\r)  \r|\\
&&\quad\quad+\lf|m_S\lf(T\lf(\lf[b-m_{\wz{S}}(b)\r]f
\chi_{{6^{N_2}B}\backslash{(6/5)S}}\r)\r)  \r|\\
&&\quad =:{\rm M}_1+{\rm M}_2+{\rm M}_3+{\rm M}_4+{\rm M}_5.
\end{eqnarray*}
By a slight modified argument similar to that used in
the proof of \cite[Theorem 7.6]{ad13}, we conclude that,
for all $x\in\cx$,
$${\rm M}_1+{\rm M}_5\ls\|b\|_{\rbmo}M_{s,(5)}f(x)$$
and
$${\rm M}_3\ls\lf[\wz K_{B,S}^{(6)}\r]^2\|b\|_
{\rbmo}M_{s,(5)}f(x).$$
To estimate ${\rm M}_2$, for $x,y\in B$, write
$$
\lf|T\lf(f\chi_{\cx\backslash{6B}}\r)(y)\r|
\leq \lf|T(f\chi_{\cx\backslash{6B}})(y)
-T(f\chi_{\cx\backslash{6B}})(x)\r|
+\lf|T(f\chi_{\cx\backslash{6B}})(x)\r|
=:\rm I+\rm II.
$$
From \eqref{3.2}, we deduce that
\begin{eqnarray*}
{\rm I}&&\leq\int_{\cx\backslash{6B}}
|K(y,z)-K(x,z)||f(z)|\,d\mu(z)\\
&&\leq\sum_{k=1}^\fz\int_{6^{k+1}B\backslash{6^kB}}
\lf|K(y,z)-K(x,z)\r||f(z)|\,d\mu(z)
\ls\|f\|_{L^{\fz}(\mu)}.
\end{eqnarray*}
The definition of $T_*$, together with \eqref{3.1},
the fact that $d(x, z)\sim d(c_B, z)$ for $x\in B$ and
$z\in \cx$ with $d(z,x)>2r_B$ and \eqref{2.2}, implies that
\begin{eqnarray*}
{\rm II}&&=\lf|\int_{\{z\in\cx:\,d(z,x)>2r_B\}}
K(x,z)f(z)\,d\mu(z)
-\int_{\{z\in 6B:\,d(z,x)>2r_B\}} K(x,z)f(z)\,d\mu(z)\r|\\
&&\ls\lf|\int_{\{z\in\cx:\,d(z,x)>2r_B\}} K(x,z)f(z)\,d\mu(z)\r|
+\int_{\{z\in 6B:\,d(z,x)>2r_B\}}
\frac1 {\lz(x,d(x,z))}|f(z)|\,d\mu(z)\\
&&\ls T_*f(x)+\int_{\{z\in 6B:\,d(z,x)>2r_B\}}
\frac1 {\lz(c_B,d(z,c_B))}|f(z)|\,d\mu(z)\\
&&\ls T_*f(x)+\frac1 {\lz(c_B,r_B)}\int_{6B} |f(z)|\,d\mu(z)
\ls T_*f(x)+M_{s,{(5)}}f(x).
\end{eqnarray*}
Thus, for $x,\,y\in B$, we have
$$
\lf|T\lf(f\chi_{\cx\backslash{6B}}\r)(y)\r|
\ls\|f\|_{L^{\fz}(\mu)}+T_*f(x)+M_{s,(5)}f(x),
$$
which, together with \eqref{2.8} and Lemma \ref{l2.9},
shows that
\begin{eqnarray*}
{\rm M}_2&&=\lf|m_B\lf(T\lf(\lf[m_{\wz{B}}(b)-m_{\wz{S}}(b)\r]
f\chi_{{\cx}\backslash{6B}}\r)\r)\r|\\
&&\ls\|b\|_{\rbmo}\wz K_{B,S}^{(6)}
\lf[\|f\|_{L^{\fz}(\mu)}+T_*f(x)+M_{s,(5)}f(x) \r].
\end{eqnarray*}
Finally, we deal with the term ${\rm M}_4$. As in the treatment
for the term ${\rm H}$ in the proof of Theorem \ref{t3.1},
an argument involving the generalization
of the H\"older inequality (see Lemma \ref{l3.3}) gives us that,
for any $y,\,z\in S$,
\begin{eqnarray*}
&&\lf|T\lf(\lf[b-m_{\wz{S}}(b)\r]
f\chi_{{\cx}\backslash{6^{N_2}B}}\r)(y)
-T\lf(\lf[b-m_{\wz{S}}(b)\r]
f\chi_{{\cx}\backslash{6^{N_2}B}}\r)(z)\r|\\
&&\hs\le\int_{{\cx}\backslash{2S}}
\lf|K(y,w)-K(z,w)\r||b(w)-m_{\wz{S}}(b)||f(w)|\,d\mu(w)\\
&&\hs\le \|f\|_{L^{\fz}(\mu)}\sum_{k=1}^{\fz}
\int_{{6^k2S}\backslash{6^{k-1}2S} }
|K(y,w)-K(z,w)||b(w)-m_{\wz{S}}(b)|\,d\mu(w)\\
&&\hs\ls\|b\|_{\rbmo}\|f\|_{L^{\fz}(\mu)}\\
&&\hs\hs\times\sum_{k=1}^{\fz}
\lf[ k\int_{{6^k2S}\backslash{6^{k-1}2S} }
|K(y,w)-K(z,w)|\,d\mu(w)+2^{-k}\r]\\
&&\hs\ls\|b\|_{\rbmo}\|f\|_{L^{\fz}(\mu)}.
\end{eqnarray*}
Taking the mean over $B$ and $S$ for $y$ and $z$,
respectively, we obtain
$${\rm M}_4\ls\|b\|_{\rbmo}\|f\|_{L^{\fz}(\mu)}.$$
Combining the estimates for ${\rm M}_1,\,{\rm M}_2,
\,{\rm M}_3,\,{\rm M}_4$ and $\rm M_5$, we obtain the
desired estimate \eqref{5.11}.

By an argument similar to that used in the proof of
\cite[Theorem 9.1]{t01} (see also the proof of
\cite[Theorem 7.6]{ad13}), together with Lemma \ref{l5.7},
\eqref{5.10} and \eqref{5.11}, we obtain
\eqref{5.9}, which completes the proof of Theorem \ref{t5.6}.
\end{proof}

We finally give the proof of Theorem \ref{t5.1}.

\begin{proof}[Proof of Theorem \ref{t5.1}]
We first show that, if the
Caler\'{o}n-Zygmund operator $T$ with kernel
satisfying \eqref{3.1} and \eqref{3.2} is bounded on $L^2(\mu)$,
then $T_*$ is bounded on $L^p(\mu)$ for all $p\in(1,\fz)$.
Indeed, Liu et al. \cite{lmy14} proved that, if $T$ with kernel
satisfying \eqref{3.1} and the \emph{H\"ormander condition},
that is, there exists a positive constant $C$ such that,
for all $x,\,\wz x\in\cx$ with $x\neq\wz x$,
$$
\dint_{d(x,y)\ge2d(x,\wz x)}[|K(x, y)-K(\wz x,y)|
+|K(y,x)-K(y,\wz x)|]\,d\mu(y)\le C,
$$
is bounded on $L^2(\mu)$, then the corresponding maximal operator
$T_*$ is bounded on $L^p(\mu)$ for all $p\in(1,\fz)$.
Since the H\"ormander-type condition \eqref{3.2} is slightly
stronger than the above H\"ormander condition, we obtain the
desired result.

On the other hand, by Lemma \ref{l3.2}(ii) and Lemma \ref{l5.4}(i),
we conclude that, for all
$p\in(1,\fz)$ and $s\in(1,p)$, $M_{s,(6)}\circ T$
is bounded on $L^p(\mu)$.

Now we assume that $b$ is bounded and consider the
following two cases for $\mu(\cx)$.

{\it Case (i)} $\mu(\cx)=\fz$. In this case, from the fact that,
for all $p\in(1,\fz)$ and $s\in(1,p)$, $M_{s,(5)}$,
$M_{s,(6)}\circ T$ and $T_*$ are bounded on $L^p(\mu)$ and
Theorems \ref{t3.1}, \ref{t5.6} and \ref{t5.5}, we deduce that
$T_b$ is bounded on $L^p(\mu)$ for all $p\in(1,\fz)$.

{\it Case (ii)} $\mu(\cx)<\fz$. In this case, by Corollary
\ref{c2.17} and the Lebesgue dominated convergence theorem,
we find that, for all $r\in(1,\,\fz)$,
\begin{equation}\label{5.12}
\lf[\dfrac1{\mu(\cx)}\dint_\cx
|b(x)-m_\cx(b)|^r\,d\mu(x)\r]^{1/r}\ls\|b\|_{\rbmo}.
\end{equation}
Write
$$|T_b f|\leq\lf|\lf[b-m_\cx(b)\r]T f\r|+
\lf|T\lf(\lf[b-m_\cx(b)\r]f\r)\r|.$$
Then, for all $p\in(1,\fz)$, from the H\"older inequality,
\eqref{5.12} and the boundedness of $T$ on $L^q(\mu)$ for
all $q\in(1,p]$, it follows that
$$\dfrac1{\mu(\cx)}\dint_\cx|T_bf(x)|\,d\mu(x)
\ls\|b\|_{\rbmo}\|f\|_{L^p(\mu)},$$
which, together with Theorems \ref{t3.1}, \ref{t5.6} and
\ref{t5.5}, implies that $T_b$ is also bounded on
$L^p(\mu)$ for all $p\in(1,\fz)$ in this case.

If $b$ is not bounded, let $q\in(0,\fz)$ and, for all
$x\in\cx$,
\begin{eqnarray*}
b_q(x):=
\begin{cases}
b(x), & {\rm if}\ |b(x)|\leq q, \\[2mm]
q\dfrac{b(x)} {\lf|b(x)\r|}, & {\rm if}\ |b(x)|>q.
\end{cases}
\end{eqnarray*}
By an argument similar to that used in the proof of
\cite[Lemma 3.11]{fyy12}, we see that $b_q\in \rbmo$ and
$\|b_q\|_{\rbmo}\ls\|b\|_{\rbmo}$, which, together with a
standard limit argument, completes the proof of Theorem
\ref{t5.1}.

\end{proof}

\bigskip

\medskip

\noindent Haibo Lin

\smallskip

\noindent College of Science, China Agricultural University,
Beijing 100083, People's Republic of China

\smallskip

\noindent{\it E-mail}: \texttt{haibolincau@126.com} (H. Lin)

\bigskip

\noindent Suqing Wu and Dachun Yang (Corresponding author)

\smallskip

\noindent School of Mathematical Sciences, Beijing Normal
University, Laboratory of Mathematics and Complex Systems,
Ministry of Education, Beijing 100875, People's Republic of
China

\smallskip

\noindent{\it E-mails}: \texttt{wusq@mail.bnu.edu.cn} (S. Wu)

\hspace{1.08cm}\texttt{dcyang@bnu.edu.cn} (D. Yang)

\bigskip

\end{document}